\documentclass[letterpaper, oneside]{amsart}
\usepackage{layout}	

\usepackage[
]{geometry}

\usepackage{amssymb}
\usepackage{amsthm}
\usepackage{amsmath}
\usepackage{enumitem}
\usepackage{mathtools}
\usepackage{hyperref}
\usepackage{colonequals}
\usepackage{mathrsfs}
\usepackage{color}
\usepackage{subcaption}
\usepackage{adjustbox}
\usepackage{ytableau}
\usepackage{bbm}

\usepackage{tikz}
\usetikzlibrary{cd}
\usetikzlibrary{decorations.markings}
\usetikzlibrary{arrows.meta}
\usetikzlibrary{patterns}

\tikzset{
    triple/.style args={[#1] in [#2] in [#3]}{
        #1,preaction={preaction={draw,#3},draw,#2}
    },
    dt/.style={
        radius=0.1
    }
}      

\usepackage{tensor}

\theoremstyle{plain}
\newtheorem{thm}{Theorem}[section]
\newtheorem{lem}[thm]{Lemma}
\newtheorem{prop}[thm]{Proposition}
\newtheorem{cor}[thm]{Corollary}

\theoremstyle{definition}
\newtheorem{defn}[thm]{Definition}
\newtheorem*{ack}{Acknowledgement}

\newtheorem{example}[thm]{Example}
\newtheorem{conj}[thm]{Conjecture}
\newtheorem{question}[thm]{Question}

\theoremstyle{remark}
\newtheorem*{rmk}{Remark}

\numberwithin{equation}{section}

\newcommand{\bC}{\mathbb{C}}
\newcommand{\bF}{\mathbb{F}}
\newcommand{\bG}{\mathbb{G}}
\newcommand{\bN}{\mathbb{N}}
\newcommand{\bQ}{\mathbb{Q}}

\newcommand{\bZ}{\mathbb{Z}}

\newcommand{\bP}{\mathbb{P}}

\newcommand{\kk}{\mathbf{k}}

\newcommand{\cB}{\mathcal{B}}

\newcommand{\cX}{\mathcal{X}}

\newcommand{\cO}{\mathcal{O}}

\newcommand{\cP}{\mathcal{P}}

\newcommand{\cN}{\mathcal{N}}

\newcommand{\fB}{\mathfrak{B}}

\newcommand{\fg}{\mathfrak{g}}

\newcommand{\fp}{\mathfrak{p}}
\newcommand{\fP}{\mathfrak{P}}

\newcommand{\vv}{\mathbbm{v}}
\newcommand{\ww}{\mathbbm{w}}

\newcommand{\Fq}{\mathbb{F}_q}

\newcommand{\Ftbar}{\overline{\mathbb{F}_2}}

\newcommand{\und}[1]{\underline{#1}}

\newcommand{\br}[1]{\left\langle{#1}\right\rangle}

\newcommand{\qlbar}{{\overline{\mathbb{Q}_\ell}}}

\newcommand{\pch}[2]{\mathord{\downarrow}^{\left(#1\right)}_{\left(#2\right)}}

\newcommand{\chicrit}{\mathbf{X}_{\textup{crit}}}

\newcommand{\tsp}[2]{\textup{\textbf{TSp}}_{#1}\left(#2\right)}
\newcommand{\tspq}[1]{\textup{\textbf{TSp}}_q\left(#1\right)}
\newcommand{\tspx}[1]{\textup{\textbf{TSp}}_x\left(#1\right)}

\newcommand{\chiq}[1]{\mathcal{X}_q\left(#1\right)}
\newcommand{\chibq}[1]{\overline{\mathcal{X}}_q\left(#1\right)}

\newcommand{\spn}[1]{\textup{span}\left(#1\right)}

\newcommand{\qe}[1]{\left\lceil{#1}\right\rfloor_q}
\newcommand{\xe}[1]{\left\lceil{#1}\right\rfloor_x}
\newcommand{\ns}{\mathnormal{ns}}

\newcommand{\orangeboxwlabel}[4]{
	\draw [line width=0.7mm] (#1,#2) rectangle (#1+#3,#2-#4);
	\draw [help lines] (#1,#2) grid (#1+#3,#2-#4);
	\foreach \x in {1,...,#3}
		\foreach \y in {1,...,#4}
			\node at (#1+\x-0.5, #2-\y+0.5) {$\vv^{\y}_{#4,\x}$};
}

\newcommand{\yellowfill}[2]{
	\fill [black!40!white] (#1) rectangle (#2);
}

\newcommand{\baseline}[1]{
	\draw [line width=0.7mm, dashed] (0-0.5,0) -- (#1+0.5,0);
}

\DeclareMathOperator{\End}{\textup{End}}

\DeclareMathOperator{\im}{\textup{im}} 

\DeclareMathOperator{\Res}{\textup{Res}}
\DeclareMathOperator{\tr}{\textup{tr}}

\DeclareMathOperator{\Ad}{\textup{Ad}}

\DeclareMathOperator{\Lie}{\textup{Lie}}

\title[On total Springer representations for $\Lie Sp_{2n}(\Ftbar)$ and the exotic case]{On total Springer representations for the symplectic Lie algebra in characteristic 2 and the exotic case}
\author{Dongkwan Kim}
\address{School of Mathematics\\
  University of Minnesota Twin Cities\\
  Minneapolis, MN 55455\\
  U.S.A.}
\email{kim00657@umn.edu}
\date{\today}							

\begin{document}
\begin{abstract} 
Let $W$ be the Weyl group of type $BC_n$.
We first provide restriction formulas of the total Springer representations for the symplectic Lie algebra in characteristic 2 and the exotic case to the maximal parabolic subgroup of $W$ which is of type $BC_{n-1}$.
Then we show that these two restriction formulas are equivalent, and discuss how the results can be used to examine the existence of affine pavings of Springer fibers corresponding to the symplectic Lie algebra in characteristic 2.\end{abstract}

\setcounter{tocdepth}{1}
\maketitle

\renewcommand\contentsname{}
\tableofcontents

\setlength{\parindent}{15pt} 	
\setlength{\parskip}{5pt} 	
\section{Introduction}
This is a continuation of the author's previous papers \cite{kim:euler, kim:total, kim:betti}. In \cite{kim:betti}, we improved the machinery in \cite{kim:euler} in order to provided a restriction formula of (graded) total Springer representations for classical types in good characteristic to a certain maximal parabolic subgroup. The main goal of this paper is to extend this method to some other cases, namely to the Springer theory of $\Lie Sp_{2n}(\Ftbar)$ and to the exotic Springer theory.

It is known that the Springer theory in bad characteristic behaves differently from good characteristic cases. For types $B$, $C$, and $D$ in characteristic 2, there are Springer theories that correspond to Lie groups, Lie algebras, and duals of Lie algebras, which are different from good characteristic cases. (Among these 9 cases, the Springer theory for the Lie groups $Sp_{2n}(\Ftbar)$ and $SO_{2n+1}(\Ftbar)$ coincide which follows from the existence of the isogeny $SO_{2n+1}(\Ftbar) \rightarrow Sp_{2n}(\Ftbar)$.) The Springer correspondences for the cases above are already known; see \cite{lus84} for classical types in good characteristic, \cite{ls85} for classical Lie groups in bad characteristic, and \cite{spa82} and \cite{xue12:comb} for classical Lie algebras and their duals in bad characteristic.

In the first part of this paper, we discuss the Springer theory of symplectic Lie algebras in characteristic 2, i.e. $\Lie Sp_{2n}(\Ftbar)$. This case is considered the simplest among the classical cases in bad characteristic, mainly since the component group of the stabilizer of any nilpotent element is connected. Also, in this case the Springer correspondence is a bijection form nilpotent orbits to irreducible characters of the Weyl group. (In general, this is only an injection even one includes the data of the component groups of the stabilizers of nilpotent elements.) This in turn causes geometry of corresponding Springer fibers, which needs to be studied in order to apply the method in \cite{kim:betti}, to be simpler than other cases in bad characteristic. (It is likely that similar argument can be applied to other cases, but it seems more complicated than $\Lie Sp_{2n}(\Ftbar)$.) Here our first main result is Theorem \ref{thm:main}, which gives an analogue of the main result of \cite{kim:betti} for $\Lie Sp_{2n}(\Ftbar)$.

In the second part, we discuss the Springer theory of the exotic nilpotent cone defined by Kato \cite{kat09}. The main result in this part is Theorem \ref{thm:mainexo}, which gives an analogue of the main result of \cite{kim:betti} in this exotic setting. The reason we study exotic Springer theory rather than other ``classical'' Springer theory is that, even though different at first sight, it shares a lot of similarities with that of $\Lie Sp_{2n}(\Ftbar)$. Resemblance of these two theories was observed in \cite{kat17} using deformation argument of an exotic nilpotent cone, and later it was revealed that they are based on the same type of combinatorics. Inspired by this point of view, the first two parts in this paper are written down to be as similar to each other as possible so that the readers can easily compare the main tools for these two theories.

In the last part, we give some remarks about our results. Firstly in Section \ref{sec:equiv} we analyze the aforementioned similarities between two Springer theories and show that our two main theorems are in fact equivalent under a certain bijection which comes from combinatorics of limit symbols studied by Shoji \cite{sho04}. In Section \ref{sec:example}, we provide some examples how the main theorems are used in small ranks. In Section \ref{sec:affinepav}, we argue how our main theorems can be used to examine the existence of affine pavings of Springer fibers for $\Lie Sp_{2n}(\Ftbar)$. (It is proved by \cite{mau17} that an affine paving exists for any exotic Springer fiber.) Then in Section \ref{sec:question} we conclude with some questions which naturally arise from our results.

\begin{ack} The author is grateful to George Lusztig for helpful comments on this topic.
\end{ack}

\section{Definitions and notations}\label{sec:defnot}
First we recall definitions and notations which are frequently used in this paper.

\subsection{Weyl groups}\label{sec:weyl} Throughout this paper we fix an integer $n \in \bZ_{>0}$. We set $W$ to be the Weyl group of type $BC_n$, which is a Coxeter group with the set of simple reflections $S=\{s_1, s_2, \ldots, s_{n}\}\subset W$ such that $(s_1s_2)^4=(s_is_{i+1})^3=id$ for $2 \leq i \leq n-1$ and $(s_i s_j)^2=id$ for any $1\leq i,j \leq n$ such that $|i-j|>1$. Set $W' \subset W$ to be the maximal parabolic subgroup of $W$ generated by $\{s_1, \ldots, s_{n-1}\}$. Then $W'$ is a Coxeter group of type $BC_{n-1}$. The main theorems in this paper are to calculate the restrictions of some total Springer representations of $W$ to $W'$.

\subsection{$\ell$-adic cohomology and geometric Frobenius}
For a variety $X$, we define $\qlbar_X$ to be the constant $\qlbar$-sheaf on $X$ where $\ell$ is a prime different from the characteristic of $X$. Also let $H^i(X) = H^i(X, \qlbar)$ be the $i$-th $\ell$-adic cohomology group of $X$ and set $H^*(X) = \sum_{i\in \bZ}(-1)^i H^i(X)$ as a virtual $\qlbar$-vector space. If a variety $X$ is defined over $\Fq$, then it is naturally equipped with a geometric Frobenius morphism $F=F_X: X \rightarrow X$, which also induces an endomorphism $F^*=F_X^* : H^i(X) \rightarrow H^i(X)$. We set $X^F$ to be the set of points in $X$ fixed by $F$.

\subsection{Partitions}\label{sec:partition} We say that $\lambda$ is a partition of $n$ and write $\lambda \vdash n$ if $\lambda=(\lambda_1, \lambda_2, \ldots, \lambda_l)$ is a finite integer sequence such that $\lambda_1\geq \lambda_2 \geq \cdots\geq \lambda_l>0$ and $\sum_{i=1}^l \lambda_i=n$. In such case we also set $|\lambda|=n$ and $l(\lambda)=l$, called the size and the length of $\lambda$, respectively. If $k>l$, then we set $\lambda_k=0$. 

We define $m_\lambda: \bZ_{>0} \rightarrow \bN$ to be a function such that $m_\lambda(r)$ is the number of parts in $\lambda$ equal to $r$. For example, if $\lambda=(6,4,4,3)$ then $m_\lambda(6)=m_\lambda(3)=1, m_\lambda(4)=2$, and $m_\lambda(r)=0$ otherwise. Also we set $m_\lambda(\geq r) \colonequals \sum_{r'\geq r}m_\lambda(r')$ and define $m_\lambda(>r), m_\lambda(\leq r), m_\lambda(<r)$ analogously. If there is no confusion, then we often omit $\lambda$ and write $m$ instead of $m_{\lambda}$. In such a case, we also write $m_r, m_{\geq r}$, etc. instead of $m(r), m(\geq r)$, etc. to simplify notations. 

Define $\und{\lambda}\subset \bZ_{>0}$ (``underlying set'') to be the set of parts in $\lambda$ (without repetition). Note that $m_\lambda(r) \neq 0$ if and only if $r \in \und{\lambda}$. For partitions $\lambda=(\lambda_1, \lambda_2, \ldots)$ and $\mu=(\mu_1, \mu_2, \ldots)$, we set $\lambda \cup \mu \vdash |\lambda|+|\mu|$ to be the partition satisfying $m_{\lambda\cup \mu}=m_\lambda+m_\mu$. In particular, we have $\und{\lambda\cup\mu} = \und{\lambda} \cup \und{\mu}$. Also we define $\lambda+\mu=(\lambda_1+\mu_1, \lambda_2+\mu_2, \ldots)$.

For $\{a_1, \ldots, a_k\} \subset \lambda$ (as a multiset) and $b_1, \ldots, b_k \in \bN$ (possibly with repetition), we set $\lambda\pch{a_1, \ldots, a_k}{b_1, \ldots, b_k}$ to be the partition obtained from $\lambda$ by substituting $a_1, \ldots, a_k$ with $b_1, \ldots, b_k$ (and reordering the result if necessary). For example, we have $(6,4,4,3)\pch{4,3}{2,2} = (6,4,2,2)$ and $(6,4,4,3)\pch{4,4}{2,0} = (6,3,2)$. Also for $1\leq a\leq b$ (resp. $1\leq a \leq b \leq l(\lambda)$), we write $\lambda^{+[a,b]}$ (resp. $\lambda^{-[a,b]}$) to be the partition obtained from $\lambda$ by increasing (resp. decreasing) the $i$-th part of $\lambda$ by 1 for each $a\leq i \leq b$ (and reordering the result if necessary). For example, we have $(4,4,3,2)^{+[2,6]} = (5,4,4,3,1,1)$, $(5,4,3,2)^{-[2,4]}= (5,3,2,1)$, and $(3,3,3,3)^{-[2,3]}= (3,3,2,2)$.

\subsection{Miscellaneous} \label{sec:misc}
For a set $X$, we denote by $\#X$ the cardinal of $X$. For a subset $X$ of a certain vector space, we write $\spn{X}$ to be the linear span of $X$ which is a vector subspace. For $a, b \in \bZ$, we define $[a,b]\colonequals \{c \in \bZ \mid a\leq c\leq b\}$. For a set $X$, we denote by $\cP(X)$ the power set of $X$. To improve readability, we write $\qe{t}$ for $q^t$ and $\xe{t}$ for $x^t$.

\part{Total Springer representations for $\Lie Sp_{2n}$ in characteristic 2}
Our goal in this part is to prove Theorem \ref{thm:main} which gives a restriction formula of the total Springer representations of $\Lie Sp_{2n}(\Ftbar)$ to $W' \subset W$.

\section{Setup}\label{sec:setup}
In this part, $\kk$ denotes the algebraic closure of $\mathbb{F}_2$. Let $q$ be a power of 2. 
Set $G$ to be the symplectic group $Sp_{2n}$ defined over $\kk$ that is split over $\Fq$. We regard $G$ as an automorphism group of a fixed $2n$-dimensional $\kk$-vector space $V$ equipped with a fixed nondegenerate symplectic form $\br{\ , \ }$, so that for any $g \in G$ and $v, w \in V$ we have $\br{gv,gw} = \br{v,w}$. For example, one may let $V = \kk^{2n}$ and define $\br{\ , \ }$ to be $\br{(x_1, x_2, \ldots, x_{2n}), (y_1, y_2, \ldots, y_{2n})}\colonequals\sum_{i=1}^{2n}x_iy_{2n-i}.$
Let $\fg$ be the Lie algebra of $G$ which is identified with $\{X \in \End(V) \mid \br{Xv,w}=\br{v, Xw} \textup{ for any } v, w \in V\}$.  We naturally identify $W$ in \ref{sec:weyl} with the Weyl group of $G$. Define $\cN \subset \fg$ to be the variety of nilpotent elements in $\fg$, called the nilpotent cone of $\fg$.

Let $\cB$ be the flag variety of $G$ which parametrizes Borel subalgebras of $\fg$. We naturally identify $\cB$ with the variety of complete isotropic flags in $V$, i.e.
$$\cB=\{F_\bullet = [0=F_0\subset F_1\subset \cdots \subset F_{2n-1} \subset F_{2n}=V] \mid \dim F_i = i, F_i^\perp=F_{2n-i} \textup{ for all } i\in [0,2n]\}.$$
For $N \in \cN$, we define $\cB_N$ to be the Springer fiber of $N$, i.e. a closed subvariety of $\cB$ consisting of all the Borel subalgebras containing $N$. Under the identification above, it also corresponds to
$\cB_N=\{F_\bullet \in \cB \mid N(F_i) \subset F_i \textup{ for all } i\in [0,2n]\}.$
Then by \cite{lus81} (which extends the method of \cite{spr76} to arbitrary characteristic) there exists an action of $W$ on $H^i(\cB_N)$ for each $i \in \bZ$, called the Springer representation. (Here we adopt the convention that $H^0(\cB_N)$ yields the trivial representation of $W$.)  If $F(N)=N$ (where $F$ is the geometric Frobenius on $\fg$ with the given $\Fq$-structure), then $F^*$ naturally acts on $H^i(\cB_N)$ and it is known that $F^*$ and the $W$-action on $H^i(\cB_N)$ commute.

\section{Nilpotent $G$-orbits}
\subsection{Parametrization of $G$-orbits in $\cN$} \label{sec:param}
Let us first describe the parametrization of orbits in $\cN$ under the adjoint action of $G$. (ref. \cite{hes79}, \cite[Section 2.6]{xue12}) Set $\Omega$ to be the set of pairs $(\lambda, \chi)$ where $\lambda$ is a partition of $2n$ and $\chi$ is a function from $\und{\lambda}$ to $\bN$ such that the following conditions hold:
\begin{enumerate}
\item if $r\in \und{\lambda}$ is odd, then $m_\lambda(r)$ is even 
\item for $r\in \und{\lambda}$, we have $0\leq \chi(r) \leq r/2$ and $\chi(r) = r/2$ if $m_\lambda(r)$ is odd
\item for $r,r' \in \und{\lambda}$ such that $r'\leq r$, we have $\chi(r')\leq \chi(r)$ and $r' - \chi(r') \leq r-\chi(r)$
\end{enumerate}
For a $G$-orbit $\cO \subset \cN$, we attach $(\lambda, \chi) \in \Omega$ as follows. Choose any $N \in \cO$ and let $\lambda\vdash 2n$ be the Jordan type of $N$. For $r\in \und{\lambda}$ we define $\chi(r) \colonequals \min\{i \in \bN \mid \br{N^{2i+1}v, v}=0 \textup{ for any } v\in \ker N^r \}.$
Then $(\lambda,\chi)$ is independent of the choice of $N \in \cO$ and this gives a bijective correspondence from the set of $G$-orbits in $\cN$ to $\Omega$.

By \cite[3.9]{spa82}, the stabilizer in $G$ of any nilpotent element in $\fg$ under the adjoint action is connected. Therefore, the Lang-Steinberg theorem implies that for any nilpotent orbit $\cO \subset \cN$, its $F$-fixed point set $\cO^F$ is nonempty and a single $G^F$-orbit. (See \cite[Chapter 3]{dm91} for more information.)

\subsection{Critical values} \label{sec:critical} For $\lambda \vdash 2n$, we set $\Omega_\lambda \colonequals \{\chi \mid (\lambda, \chi) \in \Omega\}$. Define a function $\Phi_\lambda : \Omega_\lambda \rightarrow \cP(\bZ_{>0}\times \bZ_{>0})$ where $\Phi_\lambda(\chi) \colonequals \{(r, a)\in\bZ_{>0}\times \bZ_{>0} \mid r\in \und{\lambda}, 1\leq a \leq \chi(r)\}$. It is clear that $\Phi_\lambda$ is one-to-one. There are some properties that $\Phi_\lambda(\chi)$ needs to satisfy. First, we have $\Phi_\lambda(\chi) \subset \{(i,j) \in \bZ_{>0}\times \bZ_{>0} \mid i\geq 2j\}$. Furthermore, for $r, r' \in \und{\lambda}$ such that $r<r'$, if $(r,a) \in \Phi_\lambda(\chi)$ then $(r',b) \in \Phi_\lambda(\chi)$ whenever $1\leq b\leq a$. Likewise, if $(r', b) \in \Phi_\lambda(\chi)$ then $(r,a) \in \Phi_\lambda(\chi)$ whenever $r-a\geq r'-b$. 

\begin{center}
\begin{figure}[!hbtp]
\begin{tikzpicture}
	\foreach \x in {1,...,4}
		\draw [fill] (10,\x) circle [dt];
	\foreach \x in {1,...,3}
		\draw [fill] (8,\x) circle [dt];
	\foreach \x in {1,...,3}
		\draw [fill] (6,\x) circle [dt];
	\foreach \x in {1,...,2}
		\draw [fill] (5,\x) circle [dt];
	\foreach \x in {1,...,2}
		\draw [fill] (4,\x) circle [dt];
	\foreach \x in {1,...,1}
		\draw [fill] (2,\x) circle [dt];
	\draw [help lines] (1-0.2,1-0.2) grid (10+0.2,5+0.2);
	\draw (1.6, 0.8) -- (10.2, 5.1);
	\foreach \x in {1,...,10}
		\node at (\x,1-.4) {\x};
	\foreach \y in {1,...,5}
		\node at (1-.4,\y) {\y};
	\path [draw,dashed,thick] (2-0.2,1-0.2)--(2,1) -- (3,1) -- (4,2) -- (5,2) -- (6,3) --(9,3) -- (10, 4) -- (10.2,4);
\end{tikzpicture}
\caption{An example of $\Phi_\lambda(\chi)$}
\label{fig:ex1}
\end{figure}
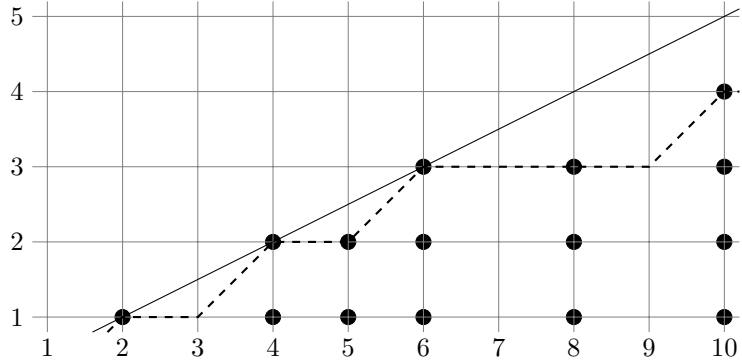
\end{center}
\begin{example}\label{ex:1} If $\lambda=(10, 10, 8,8, 6,5,5,4,4,2,2,1,1)$ and $\chi(10)=4, \chi(8)= \chi(6)=3, \chi(5)= \chi(4)=2, \chi(2)=1,\chi(1)=0$, then $\Phi_\lambda(\chi)$ is the set of points in Figure \ref{fig:ex1}. Here the aforementioned conditions imply that (1) the points in $\Phi_\lambda(\chi)$ should be either on or below the solid line, and (2) any point on or below the dashed line, with the $x$-coordinate in $\und{\lambda}$, should be contained in $\Phi_\lambda(\chi)$. Thus, in our example the set $\Phi_\lambda(\chi)$ is characterized by four points $(2,1), (4,2), (6,3), (10, 4) \in \Phi_\lambda(\chi)$. 
\end{example}

Motivated from this example, let us define a notion of critical values of $(\lambda, \chi) \in \Omega$.
\begin{defn} We say that $(\lambda, \chi)\in \Omega$ is critical at $r \in \und{\lambda}$, or $r$ is a critical value of $(\lambda, \chi)$, if
\begin{enumerate}
\item $\chi(r) \neq 0$,
\item for $r' \in \und{\lambda}$, if $r'<r$ then $\chi(r')<\chi(r)$, and 
\item for $r' \in \und{\lambda}$, if $r'>r$ then $r'-\chi(r')<r-\chi(r)$.
\end{enumerate}
We set $\chicrit^{(\lambda,\chi)}\colonequals \{(r, \chi(r)) \in \bZ_{>0}\times \bZ_{>0} \mid (\lambda, \chi) \textup{ is critical at } r\in \und{\lambda}\}$.
\end{defn}

For a partition $\lambda$ and $X \subset \bZ_{>0}\times \bZ_{>0}$, we define $\Psi_\lambda(X)$ to be a function from $\und{\lambda}$ to $\bN$ such that for $r\in \und{\lambda}$ we have
$$\Psi_\lambda(X)(r) \colonequals \max(\{r-(a-b) \mid (a,b) \in X, a\geq r\} \cup \{b \mid (a,b) \in X, a<r\} \cup \{0\}).$$
Then it is easy to verify that $\Psi_\lambda(\Phi_\lambda(\chi))=\Psi_\lambda(\chicrit^{(\lambda, \chi)})=\chi$ for any $\chi \in \Omega_\lambda$. Therefore, one may regard $\chicrit^{(\lambda, \chi)}$ as the ``minimum information'' to recover $\chi$.

\subsection{A standard model} \label{sec:standard} Here we describe a standard choice of a geometric Frobenius $F$ and a nilpotent element in a nilpotent $G$-orbit parametrized by $(\lambda, \chi) \in \Omega$. Let us fix a basis $\{\vv^{t}_{r,s} \in V\mid r \in \und{\lambda}, s\in[1, m_\lambda(r)], t \in [1,r]\}$ of $V$ and define $\br{\ , \ }$ to be
\begin{enumerate}[label=$\bullet$]
\item if $m_\lambda(r)$ is even, then $\br{\vv^{t}_{r,2k-1},\vv^{r+1-t}_{r,2k}}=1$ for $k \in [1,\frac{m_\lambda(r)}{2}]$ and $t\in [1,r]$
\item if $m_\lambda(r)$ is odd, then $\br{\vv^t_{r,1},\vv^{r+1-t}_{r,1}}=1$ and $\br{\vv^{t}_{r,2k},\vv^{r+1-t}_{r,2k+1}}=1$ for $k \in [1,\frac{m_\lambda(r)-1}{2}]$ and $t\in [1,r]$
\item $\br{\vv^{t}_{r,s}, \vv^{t'}_{r',s'}}=0$ otherwise
\end{enumerate}
and extend it to $V$ by bilinearity and (skew-)symmetry. We set the geometric Frobenius morphism $F: V \rightarrow V$ to be $F(\sum_{r,s,t}a^t_{r,s} \vv^t_{r,s}) = \sum_{r,s,t} (a^t_{r,s})^q \vv^t_{r,s}$. Then the basis $\{\vv^{t}_{r,s}\}_{r,s,t}$ and the bilinear form $\br{\ , \ }$ are defined over $\Fq$.

We define $N \in \End(V)$ to be
\begin{enumerate}[label=$\bullet$]
\item if $r \in \und{\lambda}$ is not critical or $m_\lambda(r)$ is odd, then $N\vv^t_{r,s}=\vv^{t+1}_{r,s}$ for $s\in [1, m_\lambda(r)], t\in [1, r-1]$ and $N\vv^r_{r,s}=0$
\item if $r \in \und{\lambda}$ is critical and $m_\lambda(r)$ is even, then define $N\vv^{t}_{r,s}$ to be the same as above except that $N\vv^{\chi(r)}_{r,1}= \vv^{\chi(r)+1}_{r,1}+\vv^{ r+1-\chi(r)}_{r,2}$
\end{enumerate}
and extend it to $V$ by linearity. Then one can check that $N$ is $F$-stable and contained in the nilpotent $G$-orbit parametrized by $(\lambda, \chi)$.

For $i \in \bN$, we set $\ker_{\geq i}N\colonequals \ker N \cap \im N^{i-1}$ and $\ker_{> i}N\colonequals \ker_{\geq i+1}N$. In particular, we have $\ker_{\geq i}N= \{0\}$ if $i > \lambda_1$. (Here we adopt the convention that $N^0 = Id$.) Then we have a filtration $\ker N = \ker_{\geq 1}N \supset  \ker_{\geq 2}N\supset \cdots$
and for $r \in \und{\lambda}$ we have $\ker_{\geq r}N = \spn{\vv^{r'}_{r',s} \mid r' \in \und{\lambda}, r'\geq r, s\in [1, m_\lambda(r')]}$. In particular, $\dim \ker_{\geq r}N = m_\lambda(\geq r)$.

\section{Total Springer representations and the restriction formula} 
\subsection{Total Springer representation} \label{sec:green} Recall the Springer action of $W$ on $H^i(\cB_N)$ for $N \in \cN$. Suppose that $N$ is $F$-stable and the $G$-orbit containing $N$ is parametrized by $(\lambda, \chi)$. Then the function $W \rightarrow \qlbar: w \mapsto \sum_{i \in \bZ}(-1)^i \tr(wF^*, H^i(\cB_N))$ is a character of $W$ that does not depend on the choice of $N$, which we denote by $\tspq{\lambda, \chi}$.

\begin{rmk} Note that we do \emph{not} define $\tspq{\lambda, \chi}(w)$ to be $\sum_{i \in \bZ} \tr(w, H^{2i}(\cB_N))q^i$. This is because it is not known that $\cB_N$ satisfies a certain purity condition which holds in good characteristic.
\end{rmk}



\subsection{Partial Springer resolution and the restriction formula}
We claim the following proposition which is the key step of our calculation.
\begin{prop} \label{prop:res} Suppose that $N \in \cN$ is an $F$-stable nilpotent element in the $G$-orbit parametrized by $(\lambda, \chi) \in \Omega$. Then we have
$$\Res_{W'}^W \tspq{\lambda, \chi} = \sum_{l \in \bP(\ker N)^F} \tspq{\lambda(l), \chi(l)}.$$
Here, $(\lambda(l), \chi(l))$ is the parameter of the $Sp(l^\perp/l)$-orbit containing $N|_{l^\perp/l}$, where $Sp(l^\perp/l)$ is identified with $Sp_{2n-2}$.
\end{prop}
\begin{proof}
It is essentially proved in the same way as \cite[Proposition 6.1]{kim:betti}. The proof therein is based on the result of Borho-MacPherson \cite{bm83}; there are two things to consider in order to apply their results to our setting. Firstly, their results are stated when the base field is $\bC$, i.e. in characteristic 0. However, their argument is still applicable, mutatis mutandis, to positive characteristic. (There is of no difference even when the characteristic is bad for $G$.) Secondly, their results rely on the interpretation of Springer theory in terms of  perverse sheaves, where the existence of regular semisimple elements in $\fg$ is essential. This is no longer true for $\fg=\Lie Sp_{2n}$ in characteristic 2; see e.g. \cite[13.3]{jan04}. However, since the statement only depends on the isogeny class of $\fg$, we may replace $\fg$ with the simple adjoint Lie algebra of type $C$ over $\kk$, where the existence of regular semisimple elements is guaranteed. After this, one may simply follow the proof of \cite[Proposition 6.1]{kim:betti}.
(See also \cite[Section 3]{xue12:comb} for similar argument.)
\end{proof}

\section{Calculation}\label{sec:calculation}
Let $N \in \cN$ be a nilpotent element contained in a $G$-orbit parametrized by $(\lambda, \chi) \in \Omega$. The goal of this section is to calculate the RHS of the formula in Proposition \ref{prop:res}, i.e. for each $r \in \und{\lambda}$ we calculate $\sum_{l \in (\bP(\ker_{\geq r} N)-\bP(\ker_{> r} N))^F} \tspq{\lambda(l), \chi(l)}$ using geometric argument. For simplicity we assume that $N$ and $F$ are defined as in \ref{sec:standard} and let $l =\spn{\ww}$ where $\ww=\sum_{r'\geq r, s\in[1, m_{\lambda}(r')]} a_{r',s}\vv^{r'}_{r',s}$. Then the condition $l \subset \ker_{\geq r}N - \ker_{>r}N$ is equivalent to that $a_{r,s}\neq 0$ for some $s\in [1, m_\lambda(r)]$, and $l$ is $F$-stable if and only if there exists $c\in \kk-\{0\}$ such that $ca_{r',s}\in \Fq$ for all $r' \geq r$ and $s \in [1, m_\lambda(r')]$. Now we observe the following lemma.
\begin{lem} \label{lem:lamp} Suppose that $l$ is a line contained in $\bP(\ker_{\geq r} N)-\bP(\ker_{> r} N)$. Then the Jordan type of $N|_{l^\perp/l}$ is $\lambda\pch{r,r}{r-1,r-1}$ if $\br{l, N^{-(r-1)}l}=0$ and $\lambda\pch{r}{r-2}$ otherwise.
\end{lem}
\begin{proof} It is proved similarly to characteristic 0 case, e.g. see \cite[\S 2]{sho83} or \cite[Section 5]{kim:euler}.
\end{proof}
From now on we denote $\lambda(l)$, $\chi(l)$, $\chicrit^{(\lambda, \chi)}$, $\chicrit^{(\lambda(l), \chi(l))}$, $\Psi_{\lambda}$, $\Psi_{\lambda(l)}$, $m_\lambda$, $m_{\lambda(l)}$ by $\lambda'$, $\chi'$, $\chicrit$, $\chicrit'$, $\Psi$, $\Psi'$, $m$, $m'$, respectively.  We divide all the possibilities into four cases below and argue case-by-case.
\begin{center}
\begin{tabular}{|c|c|c|c|c|}
\hline
&$r$ critical&parity of $r$&parity of $m_\lambda(r)$&$\chi(r) = r/2$\\
\hline
(\ref{sec:case1})&no&$-$&even&no\\
\hline
(\ref{sec:case2})&yes&$-$&even&no\\
\hline
(\ref{sec:case3})&yes&even&odd&yes\\
\hline
(\ref{sec:case4})&yes&even&even&yes\\
\hline
\end{tabular}
\end{center}


\subsection{$r$ is not critical for $(\lambda, \chi)$} \label{sec:case1} (It corresponds to $r\in \{1, 5, 8\}$ case in Example \ref{ex:1}.) We necessarily have $\chi(r)<r/2$ and also $m(r)$ is even. In this case, we always have $\lambda'=\lambda\pch{r,r}{r-1,r-1}$ by Lemma \ref{lem:lamp} since $\br{l, N^{-(r-1)}l}=0$. On the other hand, we always have $\chicrit' \supset \chicrit$; the critical values for $(\lambda,\chi)$ remain critical for $(\lambda', \chi')$. More precisely, there are two cases to consider.

\begin{enumerate}[label={\bfseries Case \arabic*.}, leftmargin=\parindent, itemindent=\labelwidth]
\item First suppose that either $\chi(r)=0$ or there exists $\tilde{r}\in \und{\lambda}$ such that $\tilde{r}<r$ and $\chi(\tilde{r})=\chi(r)$. Then direct calculation shows that we always have $\chicrit'= \chicrit$. As $\#(\bP(\ker_{\geq r}N)-\bP(\ker_{>r} N))^F =  \frac{\qe{m(\geq r)}-\qe{m(>r)}}{q-1}$, in this case we have
$$\sum_{l \in (\bP(\ker_{\geq r} N)-\bP(\ker_{> r} N))^F} \tspq{\lambda', \chi'}= \frac{\qe{m(\geq r)}-\qe{m(>r)}}{q-1} \tspq{\lambda\pch{r,r}{r-1,r-1}, \Psi'\left(\chicrit\right)}.$$
\item Suppose otherwise. Since $r$ is not critical, it means that there exists $\tilde{r} \in \und{\lambda}$ such that $\tilde{r}>r$ and $\tilde{r}-\chi(\tilde{r}) = r-\chi(r)$. We have two possibilities:
\begin{enumerate}
\item $\chicrit'= \chicrit \cup \{(r-1, \chi(r))\}$ if there exists $v\in N^{-(r-1)}l$ such that $\br{N^{2\chi(r)-1}v,v}\neq 0$
\item $\chicrit'= \chicrit$ otherwise
\end{enumerate}
By direct calculation, one can show that the first situation happens if and only if $a_{\tilde{r},1} \neq 0$. (Recall that $l= \spn{\sum_{r'\geq r, s\in [1, m(r')]} a_{r',s}\vv^{r'}_{r',s}}$.) If we set $H\subset \bP(\ker_{\geq r}N)$ to be the hyperplane defined by the equation $a_{\tilde{r},1}=0$, then it is defined over $\Fq$ and it intersects $\bP(\ker_{>r}N)\subset \bP(\ker_{\geq r}N)$ transversally. As $\#(\bP(\ker_{\geq r}N)-\bP(\ker_{>r}N)\cup H)^F =\qe{m(\geq r)-1}-\qe{m(>r)-1}$ and $\#(H-\bP(\ker_{>r}N))^F=\frac{\qe{m(\geq r)-1}-\qe{m(>r)-1}}{q-1}$, we have
\begin{align*}
&\sum_{l \in (\bP(\ker_{\geq r} N)-\bP(\ker_{> r} N))^F} \tspq{\lambda', \chi'}
\\&=\frac{\qe{m(\geq r)-1}-\qe{m(>r)-1}}{q-1} \tspq{\lambda\pch{r,r}{r-1,r-1}, \Psi'\left(\chicrit\right)}
\\&\quad +(\qe{m(\geq r)-1}-\qe{m(>r)-1}) \tspq{\lambda\pch{r,r}{r-1,r-1},   \Psi'\left( \chicrit \cup \{(r-1, \chi(r))\}\right)}.
\end{align*}
\end{enumerate}
Note that if we are in Case 1 then $\Psi'\left(\chicrit\right)=\Psi'( \chicrit \cup \{(r-1, \chi(r))\})$ even if $r-1$ is not critical for $(\lambda', \chi')$, in which case the RHS of the formula in Case 2 coincides with that in Case 1. Thus in either case we may use the formula in Case 2 to calculate $\sum_{l \in (\bP(\ker_{\geq r} N)-\bP(\ker_{> r} N))^F} \tspq{\lambda', \chi'}$.
\begin{rmk} Scheme-theoretically, the hyperplane $H$ in Case 2 should be defined by the equation $(a_{\tilde{r},1})^2=0$, in which case $H$ is not reduced. However, this does not cause any problem as we only deal with $\ell$-adic cohomology. (Similar phenomena occur in other cases as well.)
\end{rmk}


\subsection{$r$ is critical and $\chi(r) \neq r/2$} \label{sec:case2} (It corresponds to $r=10$ case in Example \ref{ex:1}.) The condition forces that $m(r)$ is even and $\chi(r-1) =\chi(r)-1$ if $r-1 \in \und{\lambda}$. In this case we always have $\lambda'=\lambda\pch{r,r}{r-1,r-1}$ and $\chi'(r-1) \in \{\chi(r),\chi(r-1)\}$ similarly to the first case. Namely, let us set $H\subset \bP(\ker_{\geq r} N)$ to be the hyperplane defined by $a_{r,1}=0$ which contains $\bP(\ker_{>r}N)$. If $l\not\in H$ then there exists $v\in N^{-(r-1)}l\cap l^\perp$ such that $\br{N^{2\chi(r)-1}v,v}\neq 0$, from which we have $\chi'(r-1) = \chi(r)$. (One may choose $v\in N^{-(r-1)}l\cap l^\perp$ to be $\sum_{r'\geq r, s\in [1, m(r')]} a_{r',s}\vv^{r'-(r-1)}_{r',s}$, in which case $\br{N^{2\chi(r)-1}v,v}= a_{r,1}^2\neq 0$.) As  $\#(\bP(\ker_{\geq r}N)-H)^F=\qe{m(\geq r)-1}$, we have 
$$\sum_{l \in (\bP(\ker_{\geq r}N)-H)^F} \tspq{\lambda',\chi'} = \qe{m(\geq r)-1}\tspq{\lambda\pch{r,r}{r-1,r-1}, \Psi'(\chicrit \cup \{(r-1, \chi(r))\})}.$$

%

Now suppose that $l\in H-\bP(\ker_{>r}N)$, in which case $\chi'(r-1) = \chi(r) -1$. First assume that $m(r)>2$ so that $\lambda' \owns r$. Then $\chi'(r) = \chi(r)$ if and only if there exists $v\in l^\perp \cap \ker N^r$ such that $\br{N^{2\chi(r)-1}v,v}\neq 0$. Let us set $H' \subset H$ to be a linear subvariety defined by ($a_{r,1}=0$ and) $a_{r,3}=\cdots = a_{r, m(r)}=0$ which contains $\bP(\ker_{>r}N)$. If $l \in H-H'$ then we can always find $v=\vv^1_{r,1}+\sum_{s=3}^{m(r)}b_{r,s}\vv^1_{r,s}$ for some $\{b_{r,s}\}$, not all zero, such that $\br{v, l}=0$, $N^rv=0$, and $\br{N^{2\chi(r)-1}v,v}\neq 0$. (The existence of such  $\{b_{r,s}\}$ follows from the fact that the pairing $\br{\ , \ } : \spn{\vv_{r,s}^1 \mid s \in [3,m(r)]} \times \spn{\vv_{r,s}^r \mid s \in [3,m(r)]} \rightarrow \kk$ is perfect.) Thus we have $\chi'(r)=\chi(r)$ and as a result $\chicrit'=\chicrit$. As $(H-H')^F=\frac{\qe{m({\geq r})-1} -\qe{m({>r})+1}}{q-1}$, it follows that 
$$\sum_{l \in (H-H')^F} \tspq{\lambda',\chi'} =  \frac{\qe{m({\geq r})-1} -\qe{m({>r})+1}}{q-1}\tspq{\lambda\pch{r,r}{r-1,r-1}, \Psi'(\chicrit)}.$$
Meanwhile if $m(r)=2$ then $r \not\in \und{\lambda'}$ and also the coefficient of the RHS in the formula above is zero. In this case we simply ignore this term.

We suppose that $l \in H'$. Then necessarily $a_{r,2}\neq0$, which forces that $N^{-(r-1)}\vv_{r,1} \cap l^\perp=\emptyset$. From this one can easily deduce that $\chi'(r) = \chi(r)-1$. Also, we may assume that $a_{r,2}=1$, i.e. $\ww-\vv_{r,2} \in \ker_{>r}N$. Note that the map $l \mapsto \ww-\vv_{r,2}$ defines an isomorphism of varieties $H' - \bP(\ker_{>r}N)\simeq \ker_{>r}N$. Then there exists a unique $j \in \und{\lambda}$ such that $j>r$ and $\ww-\vv_{r,2} \in \ker_{\geq j}N - \ker_{>j}N$. In this case, direct calculation shows that there exists $v \in \ker N^j \cap l^\perp$ such that $\br{N^{2\chi(r)-1}v,v}\neq 0$, which implies that $\chi'(j) \geq  \chi(r)$. Here, it is crucially used that the pairing $\br{\ , \ }: \spn{\vv^{j}_{j,s} \mid s\in [1, m(j)]} \times \spn{\vv^{1}_{j,s} \mid s\in [1, m(j)]} \rightarrow \kk$ is perfect, which in particular implies that for any $w \in \ker_{\geq j}N - \ker_{>j}N$ there exists $v \in \ker N^j-\ker N^{j-1}$ such that $\br{v,w} \neq 0$.

From this observation we obtain the following result. Let $\chicrit^*=\chicrit\cup  \{(r-1, \chi(r)-1)\}-\{(r,\chi(r)\}$ and let $j\in \und{\lambda}$ be the largest value such that $\chi(j) = \chi(r)$ and $j-\chi(r) \neq r' -\chi(r')$ for any $r'\in \und{\lambda}$ such that $r'>j$. (The last condition is imposed so that $\Psi'(\chicrit^* \cup \{(j, \chi(r))\}) \neq \Psi'(\chicrit^*)$.)  Set $\und{\lambda} \cap [r,j] = \{j_1, j_2, \ldots, j_{a+1}\}$ where $j=j_1>j_2>\cdots>j_{a+1}=r$. Then $a\geq 0$ since $r$ is critical, and $\chi(j_1) = \cdots = \chi(j_{a+1})=\chi(r)$. Also, we have
\begin{enumerate}
\item if $\ww-\vv_{r,2}  \in \ker_{>j}N$, then $\chicrit' = \chicrit^*$
\item if $\ww-\vv_{r,2} \in \ker_{> j_{b+1}} N- \ker_{> j_b}N$ for some $1\leq b \leq a$, then $\chicrit'=\chicrit^* \cup \{(j_b, \chi(r))\}$
\end{enumerate}
Therefore, it follows that
\begin{align*}
&\sum_{l \in (H' - \bP(\ker_{>r}N))^F} \tspq{\lambda',\chi'} =  \qe{m(>j)}\tspq{\lambda\pch{r,r}{r-1,r-1}, \Psi'(\chicrit^*)}
\\&\qquad \qquad +\sum_{b=1}^{a} (\qe{m(\geq j_b)}-\qe{m(>j_b)})\tspq{\lambda\pch{r,r}{r-1,r-1}, \Psi'(\chicrit^* \cup \{(j_b, \chi(r))\})}.
\end{align*}


\begin{rmk} As this case shows, it is not possible to use a method similar to \cite[Section 7, 8]{kim:betti}, i.e. dividing $V$ into orthogonal pieces each of which corresponds to a ``rectangle'' in the Jordan type of $N$ and adding up the outcomes. One possible explanation of this phenomenon is that one cannot directly use \cite[Lemma 8.1]{kim:betti} whose proof uses division by 2.
\end{rmk}


\subsection{$r$ is critical, $\chi(r) = r/2$, and $m(r)$ is odd} \label{sec:case3}
(It corresponds to $r=6$ case in Example \ref{ex:1}.)
Clearly $r$ is even. Let us set $H\subset\bP(\ker_{\geq r}N)$ to be the hyperplane defined by $a_{r,1}=0$ which contains $\bP(\ker_{>r}N)$. If $m(r)>1$ and $l \in H - \bP(\ker_{>r}N)$ then $\lambda' = \lambda\pch{r,r}{r-1,r-1}$ and direct calculation shows that $\chicrit'=\chicrit$. As $\#(H-\bP(\ker_{>r} N))^F = \frac{\qe{m(\geq r)-1}-\qe{m(>r)}}{q-1}$, we have
$$\sum_{l \in (H - \bP(\ker_{>r}N))^F} \tspq{\lambda',\chi'} =  \frac{\qe{m(\geq r)-1}-\qe{m(>r)}}{q-1}\tspq{\lambda\pch{r,r}{r-1,r-1}, \Psi'(\chicrit)}.$$
If $m(r)=1$ then the coefficient of the RHS above is zero, which is consistent with that $H= \bP(\ker_{>r}N)$, i.e. $H -\bP(\ker_{>r}N)=\emptyset$. In such a case we simply ignore this term.

This time suppose $l \not\in H$. Then there exists $v \in N^{-(r-1)}(l)$ such that $\br{v, N^{r-1}v}\neq 0$, which means that $\lambda' = \lambda\pch{r}{r-2}$. (One may take $v=\sum_{r'\geq r, s\in [1, m(r')]} a_{r',s}\vv^{r'-(r-1)}_{r',s}$.) Furthermore, there exists $v\in N^{-(r-2)}l\cap l^\perp$ such that $\br{N^{r-3}v, v} \neq 0$, which implies that $\chi'(r-2) = \chi(r)-1 = (r-2)/2$. (Similarly to above one may take $v=\sum_{r'\geq r, s\in [1, m(r')]} a_{r',s}\vv^{r'-(r-2)}_{r',s}$.) Then $r-2$ is always a critical value of $(\lambda', \chi')$, i.e. $(r-2, (r-2)/2) \in \chicrit'$.


After this, we argue similarly to the second case. We set $H' \subset \bP(\ker_{\geq r}N)$ to be a linear subvariety defined by $a_{r,2} = a_{r,3}=\cdots=a_{r,m(r)}=0$. Then $H' \cap H = \bP(\ker_{>r}N)$, and if $l \not \in H' \cup H$ then one can show that $\chi'(r)=\chi(r)=r/2$, which in turn implies that $\chi'= \Psi'(\chicrit \cup \{(r-2, (r-2)/2)\})$. As $\#(\bP(\ker_{\geq r}N) - H'-H)^F = \qe{m(\geq r)-1}-\qe{m(>r)}$, it follows that 
\begin{align*}
&\sum_{l \in (\bP(\ker_{\geq r}N) - H'-H)^F} \tspq{\lambda',\chi'} 
\\&=  (\qe{m(\geq r)-1}-\qe{m(>r)})\tspq{\lambda\pch{r}{r-2},\Psi'(\chicrit \cup \{(r-2, \frac{r-2}{2})\})}.
\end{align*}

Now we suppose that $l \in H'$. Then $a_{r,1}\neq0$, which forces that $N^{-(r-1)}\vv_{r,1} \cap l^\perp=\emptyset$. From this one can easily deduce that $\chi'(r) = \chi(r)-1$. Also, we may assume that $a_{r,1}=1$, i.e. $\ww-\vv_{r,2} \in \ker_{>r}N$. Note that the map $l \mapsto \ww-\vv_{r,2}$ defines an isomorphism of varieties $H' - \bP(\ker_{>r}N)\simeq \ker_{>r}N$.  Then there exists a unique $j \in \und{\lambda}$ such that $j>r$ and $w \in \ker_{\geq j}N - \ker_{>j}N$. In this case, direct calculation shows that there exists $v \in \ker N^j \cap l^\perp$ such that $\br{N^{2\chi(r)-1}v,v}\neq 0$, which implies that $\chi'(j)\geq \chi(r)$. Here, it is crucially used that the pairing $\br{\ , \ }: \spn{\vv^{j}_{j,s} \mid s\in [1, m(j)]} \times \spn{\vv^{1}_{j,s} \mid s\in [1, m(j)]} \rightarrow \kk$ is perfect, which in particular implies that for any $w \in \ker_{\geq j}N - \ker_{>j}N$ there exists $v \in \ker N^j-\ker N^{j-1}$ such that $\br{v,w} \neq 0$.

From this observation we obtain the following result. Set $\chicrit^{**}=\chicrit\cup  \{(r-2, (r-2)/2)\}-\{(r,r/2)\}$ and let $j\in \und{\lambda}$ be the largest value such that $\chi(j) = \chi(r)=r/2$ and $j-\chi(r) \neq r' -\chi(r')$ for any $r'\in \und{\lambda}$ such that $r'>j$. (The last condition is imposed so that $\Psi'(\chicrit^{**} \cup \{(j_1, \chi(r))\}) \neq \Psi'(\chicrit^{**})$.)  Set $\und{\lambda} \cap [r,j] = \{j_1, j_2, \ldots, j_{a+1}\}$ where $j=j_1>j_2>\cdots>j_{a+1}=r$. Then $a\geq 0$ since $r$ is critical, and $\chi(j_1) = \cdots = \chi(j_{a+1})=\chi(r)=r/2$. Also, we have
\begin{enumerate}
\item if $w \in \ker_{>j}N$, then $\chicrit' = \chicrit^{**}$
\item if $w \in \ker_{> j_{b+1}} N- \ker_{> j_b}N$ for some $1\leq s \leq a$, then $\chicrit'=\chicrit^{**} \cup \{(j_b, r/2)\}$
\end{enumerate}
Therefore, it follows that
\begin{align*}
&\sum_{l \in (H' - \bP(\ker_{>r}N))^F} \tspq{\lambda',\chi'} =  \qe{m(>j)}\tspq{\lambda\pch{r}{r-2}, \Psi'(\chicrit^{**})}
\\&\qquad \qquad +\sum_{b=1}^{a} (\qe{m(\geq j_b)}-\qe{m(>j_b)})\tspq{\lambda\pch{r}{r-2}, \Psi'(\chicrit^{**} \cup \{(j_b, r/2)\})}.
\end{align*}


%

\subsection{$r$ is critical, $\chi(r)=r/2$, and $m(r)$ is even}\label{sec:case4} 
(It corresponds to $r\in \{2, 4\}$ case in Example \ref{ex:1}.)
In this case $r$ is still even. Let us set $H \subset \bP(\ker_{\geq r}N)$ to be the hyperplane defined by the equation $a_{r,1}=0$ which contains $\bP(\ker_{>r}N)$. If $l \not\in H$ then there exists $v \in N^{-(r-1)}(l)$ such that $\br{v, N^{r-1}v} \neq 0$, thus $\lambda' = \lambda\pch{r}{r-2}$. (One may take $v=\sum_{r'\geq r, s\in[1, m(r')]} a_{r',s}\vv^{r'-(r-1)}_{r',s}$.)  Furthermore, there exists $v \in N^{-(r-2)}l \cap l^\perp$ such that $\br{v, N^{r-3}v}\neq 0$, which means that $\chi'(r-2) = \chi(r) - 1 = (r-2)/2$. (Similarly one may take $v=\sum_{r'\geq r, s\in[1,m(r')]} a_{r',s}\vv^{r'-(r-2)}_{r',s}$.)
As $\#(\bP(\ker_{\geq r}N)-H)^F = \qe{m(\geq r)-1}$, we have
$$\sum_{l \in (\bP(\ker_{\geq r}N)-H)^F} \tspq{\lambda',\chi'} = \qe{m(\geq r)-1}\tspq{\lambda\pch{r}{r-2}, \Psi'(\chicrit \cup \{(r-2, \frac{r-2}{2})\})}.$$

After this, the argument here is similar to the second and the third cases above. Suppose that $l \in H$. Then we have $\lambda'= \lambda\pch{r,r}{r-1,r-1}$. Furthermore, there exists $v\in N^{-(r-1)}l \cap l^{\perp}$ such that $\br{N^{r-3}v,v} \neq0$, which means that $\chi'(r-1) = \chi(r)-1 = (r-2)/2$.  (One may take $v=\sum_{r'\geq r, s\in[1, m(r')]} a_{r',s}\vv^{r'-(r-1)}_{r',s}$.) 

Set $H' \subset H$ to be a linear subvariety defined by ($a_{r,1}=0$ and) $a_{r,3}=a_{r,4}=\cdots=a_{r,m(r)}=0$ which contains $\bP(\ker_{>r}N)$. If $m(r)>2$ and $l\in H-H'$, then one can show that $\chi'(r) = \chi(r) = r/2$, which follows that $\chi'=\Psi'(\chicrit)$. As $\#(H-H')^F = \frac{\qe{m(\geq r)-1}-\qe{m(>r)+1}}{q-1}$, we have
$$\sum_{l \in (H-H')^F} \tspq{\lambda',\chi'} = \frac{\qe{m(\geq r)-1}-\qe{m(>r)+1}}{q-1}\tspq{\lambda\pch{r,r}{r-1,r-1}, \Psi'(\chicrit)}.$$
When $m(r)=2$, then the coefficient of the RHS is zero, which is consistent with the fact that $H=H'$, i.e. $H-H'=\emptyset$. In this case we simply ignore this term.




Now we suppose that $l \in H'$. Then $a_{r,2}\neq0$, which forces that $N^{-(r-1)}\vv_{r,1} \cap l^\perp=\emptyset$. From this one can easily deduce that $\chi'(r) = \chi(r)-1$. Also, we may assume that $a_{r,2}=1$, i.e. $\ww-\vv_{r,2}^r \in \ker_{>r}N$. Note that the map $l \mapsto \ww-\vv_{r,2}^r$ defines an isomorphism of varieties $H' - \bP(\ker_{>r}N) \simeq \ker_{>r}N$. Then there exists a unique $j \in \und{\lambda}$ such that $j>r$ and $\ww-\vv_{r,2}^r \in \ker_{\geq j}N - \ker_{>j}N$. In this case, direct calculation shows that there exists $v \in \ker N^j \cap l^\perp$ such that $\br{N^{2\chi(r)-1}v,v}\neq 0$, which implies that $\chi'(j)\geq \chi(r)$. Here, it is crucially used that the pairing $\br{\ , \ }: \spn{\vv^{j}_{j,s} \mid s\in [1, m(j)]} \times \spn{\vv^{1}_{j,s} \mid s\in [1, m(j)]} \rightarrow \kk$ is perfect, which in particular implies that for any $w \in \ker_{\geq j}N - \ker_{>j}N$ there exists $v \in \ker N^j-\ker N^{j-1}$ such that $\br{v,w} \neq 0$.

From this observation we obtain the following result. Set $\chicrit^{***}=\chicrit\cup  \{(r-1, (r-2)/2)\}-\{(r,r/2)\}$ and let $j\in \und{\lambda}$ be the largest value such that $\chi(j) = \chi(r)=r/2$ and $j-\chi(r) \neq r' -\chi(r')$ for any $r'\in \und{\lambda}$ such that $r'>j$. (The last condition is imposed so that $\Psi'(\chicrit^{***} \cup \{(j_1, \chi(r))\}) \neq \Psi'(\chicrit^{***})$.)  Set $\und{\lambda} \cap [r,j] = \{j_1, j_2, \ldots, j_{a+1}\}$ where $j=j_1>j_2>\cdots>j_{a+1}=r$. Then $a\geq 0$ since $r$ is critical, and $\chi(j_1) = \cdots = \chi(j_{a+1})=\chi(r)=r/2$. Also, we have
\begin{enumerate}
\item if $w \in \ker_{>j}N$, then $\chicrit' = \chicrit^{***}$
\item if $w \in \ker_{> j_{b+1}} N- \ker_{> j_b}N$ for some $1\leq b \leq a$, then $\chicrit'=\chicrit^{***} \cup \{(j_b, r/2)\}$
\end{enumerate}
Therefore, it follows that
\begin{align*}
&\sum_{l \in (H' - \bP(\ker_{>r}N))^F} \tspq{\lambda',\chi'} =  \qe{m(>j)}\tspq{\lambda\pch{r,r}{r-1,r-1}, \Psi'(\chicrit^*)}
\\&\qquad \qquad +\sum_{b=1}^{a} (\qe{m(\geq j_b)}-\qe{m(>j_b)})\tspq{\lambda\pch{r,r}{r-1,r-1}, \Psi'(\chicrit^* \cup \{(j_b, r/2)\})}.
\end{align*}

\section{Main theorem}
We summarize the results in the previous section and conclude our first main theorem. First, we recall some notations; see \ref{sec:weyl} for $W$ and $W'$; see \ref{sec:partition} for $\und{\lambda}$, $\lambda\pch{a_1, a_2, \ldots}{b_1, b_2, \ldots}$,  $m(r)$, $m_{\geq r}$, etc.; see \ref{sec:misc} for $\qe{-}$; see \ref{sec:param} for $\Omega$; see \ref{sec:critical} for critical values, $\Psi'$, and $\chicrit$ (here, $\Psi'$ is either $\Psi_{\lambda\pch{r,r}{r-1,r-1}}$ or $\Psi_{\lambda\pch{r}{r-2}}$ depending on each term); see \ref{sec:green} for $\tspq{\lambda, \chi}$.
\begin{thm}[Main theorem for the symplectic Lie algebra in characteristic 2] \label{thm:main} For $(\lambda, \chi)\in \Omega$, the character $\Res^W_{W'} \tspq{\lambda, \chi}$ is equal to
\begin{align*}
\sum_{\substack{r\in\und{\lambda} \\\textup{ not critical}}}&
\left[
\begin{aligned}
&\frac{\qe{m_{\geq r}-1}-\qe{m_{>r}-1}}{q-1} \tspq{\lambda\pch{r,r}{r-1,r-1}, \Psi'\left(\chicrit\right)}
\\&+\left(\qe{m_{\geq r}-1}-\qe{m_{>r}-1}\right) \tspq{\lambda\pch{r,r}{r-1,r-1},   \Psi'\left( \chicrit \cup \{(r-1, \chi(r))\}\right)}
\end{aligned}\right]\allowdisplaybreaks
\\+\sum_{\substack{r\in\und{\lambda} \textup{ critical,} \\ \chi(r) \neq r/2}}&
\left[
\begin{aligned}
&\qe{m_{\geq r}-1}\tspq{\lambda\pch{r,r}{r-1,r-1}, \Psi'(\chicrit \cup \{(r-1, \chi(r))\})}
\\&+\frac{\qe{m_{\geq r}-1} -\qe{m_{>r}+1}}{q-1}\tspq{\lambda\pch{r,r}{r-1,r-1}, \Psi'(\chicrit)}
\\&+ \qe{m_{>j}}\tspq{\lambda\pch{r,r}{r-1,r-1}, \Psi'(\chicrit^*)}
\\&+\sum_{k\in \und{\lambda}, r<k\leq j} \left(\qe{m_{\geq k}}-\qe{m_{>k}}\right)\tspq{\lambda\pch{r,r}{r-1,r-1}, \Psi'(\chicrit^* \cup \{(k, \chi(r))\})}
\end{aligned}\right]\allowdisplaybreaks
\\+\sum_{\substack{r\in\und{\lambda} \textup{ critical,} \\ \chi(r) = r/2, \\ m(r) \textup{ odd}}}&
\left[
\begin{aligned}
&\frac{\qe{m_{\geq r}-1}-\qe{m_{>r}}}{q-1}\tspq{\lambda\pch{r,r}{r-1,r-1}, \Psi'(\chicrit)}
\\&+\left(\qe{m_{\geq r}-1}-\qe{m_{>r}}\right)\tspq{\lambda\pch{r}{r-2},\Psi'(\chicrit \cup \{(r-2, (r-2)/2)\})}
\\&+\qe{m_{>j}}\tspq{\lambda\pch{r}{r-2}, \Psi'(\chicrit^{**})}
\\&+\sum_{k\in \und{\lambda}, r<k\leq j}  \left(\qe{m_{\geq k}}-\qe{m_{>k}}\right)\tspq{\lambda\pch{r}{r-2}, \Psi'(\chicrit^{**} \cup \{(k, r/2)\})}
\end{aligned}\right]\allowdisplaybreaks
\\+\sum_{\substack{r\in\und{\lambda} \textup{ critical,} \\ \chi(r) = r/2, \\ m(r) \textup{ even}}}&
\left[
\begin{aligned}
& \qe{m_{\geq r}-1}\tspq{\lambda\pch{r}{r-2}, \Psi'(\chicrit \cup \{(r-2, \frac{r-2}{2})\})}
\\&+\frac{\qe{m_{\geq r}-1}-\qe{m_{>r}+1}}{q-1}\tspq{\lambda\pch{r,r}{r-1,r-1}, \Psi'(\chicrit)}
\\&+ \qe{m_{>j}}\tspq{\lambda\pch{r,r}{r-1,r-1}, \Psi'(\chicrit^{***})}
\\&+\sum_{k\in \und{\lambda}, r<k\leq j}  \left(\qe{m_{\geq k}}-\qe{m_{>k}}\right)\tspq{\lambda\pch{r,r}{r-1,r-1}, \Psi'(\chicrit^{***} \cup \{(k, r/2)\})}
\end{aligned}\right].
\end{align*}
Here, $m(r)$, $m_{\geq r}$, etc. are defined with respect to $\lambda$.  Also, we set 
\begin{enumerate}[label=$-$]
\item $\chicrit^*=\chicrit\cup  \{(r-1, \chi(r)-1)\}-\{(r,\chi(r))\}$, 
\item $\chicrit^{**}=\chicrit\cup  \{(r-2, (r-2)/2)\}-\{(r,r/2)\}$, and 
\item $\chicrit^{***}=\chicrit\cup  \{(r-1, (r-2)/2)\}-\{(r,r/2)\}$.
\end{enumerate}
When $r$ is critical, we set $j \in \und{\lambda}$ to be the largest value such that $\chi(j)=\chi(r)$ and $r'-\chi(r')\neq j-\chi(r)$ for any $r'\in \und{\lambda}$ such that $r'>j$. (Such $j$ always exists and may equal $r$.)
We ignore terms whose coefficients are zero.
\end{thm}

If we evaluate the above theorem at $q=1$, then we have an ungraded version which is an analogue of \cite[Theorem 5.3]{kim:euler} for $\Lie Sp_{2n}(\Ftbar)$.
\begin{cor} Let $\tsp{1}{-,-}\colonequals \tspq{-,-}|_{q=1}$. Then for $(\lambda, \chi)\in \Omega$, the character $\Res^W_{W'} \tsp{1}{\lambda, \chi}$ is equal to
\begin{align*}
\sum_{\substack{r\in\und{\lambda} \\\textup{ not critical}}}&
m_r \tsp{1}{\lambda\pch{r,r}{r-1,r-1}, \Psi'\left(\chicrit\right)}
\allowdisplaybreaks
\\+\sum_{\substack{r\in\und{\lambda} \textup{ critical,} \\ \chi(r) \neq r/2}}&
\left[
\begin{aligned}
&\tsp{1}{\lambda\pch{r,r}{r-1,r-1}, \Psi'(\chicrit \cup \{(r-1, \chi(r))\})}
\\&+(m_r-2)\tsp{1}{\lambda\pch{r,r}{r-1,r-1}, \Psi'(\chicrit)}
+ \tsp{1}{\lambda\pch{r,r}{r-1,r-1}, \Psi'(\chicrit^*)}
\end{aligned}\right]\allowdisplaybreaks
\\+\sum_{\substack{r\in\und{\lambda} \textup{ critical,} \\ \chi(r) = r/2, \\ m(r) \textup{ odd}}}&
\left(
(m_r-1)\tsp{1}{\lambda\pch{r,r}{r-1,r-1}, \Psi'(\chicrit)}
+\tsp{1}{\lambda\pch{r}{r-2}, \Psi'(\chicrit^{**})}
\right)\allowdisplaybreaks
\\+\sum_{\substack{r\in\und{\lambda} \textup{ critical,} \\ \chi(r) = r/2, \\ m(r) \textup{ even}}}&
\left[
\begin{aligned}
& \tsp{1}{\lambda\pch{r}{r-2}, \Psi'(\chicrit \cup \{(r-2, \frac{r-2}{2})\})}
\\&+(m_r-2)\tsp{1}{\lambda\pch{r,r}{r-1,r-1}, \Psi'(\chicrit)}+\tsp{1}{\lambda\pch{r,r}{r-1,r-1}, \Psi'(\chicrit^{***})}
\end{aligned}\right].
\end{align*}
Here, $m(r)$, $m_r$, etc. are defined with respect to $\lambda$.  Also, we set 
\begin{enumerate}[label=$-$]
\item $\chicrit^*=\chicrit\cup  \{(r-1, \chi(r)-1)\}-\{(r,\chi(r))\}$, 
\item $\chicrit^{**}=\chicrit\cup  \{(r-2, (r-2)/2)\}-\{(r,r/2)\}$, and 
\item $\chicrit^{***}=\chicrit\cup  \{(r-1, (r-2)/2)\}-\{(r,r/2)\}$.
\end{enumerate}
We ignore terms whose coefficients are zero.
\end{cor}

\part{Total Springer representations in the exotic case} \label{part:exotic}
Our goal in this part is to prove Theorem \ref{thm:mainexo} which gives a restriction formula of the total Springer representations in the exotic case. The structure of this part is, mutatis mutandis, the same as the previous part.

\section{Setup}
In this part $\kk$ is the algebraic closure of $\mathbb{F}_p$ where $p$ is an odd prime number. Let $q$ be a power of $p$. Let $G$ be the general linear group $GL_{2n}$ over $\kk$ which is split over $\Fq$ and let $\fg$ be its Lie algebra. We regard $G$ as an automorphism group of a fixed $2n$-dimensional $\kk$-vector space $V$. Then $\fg$ is naturally identified with the endomorphism Lie algebra of $V$.

Let $\br{\ , \ }$ be a fixed symplectic form on $V$ defined over $\Fq$. This induces an involutive automorphism $\theta: G\rightarrow G$ such that for any $g\in G$ and $v, w \in V$, we have $\br{g^{-1}v,w} = \br{v, \theta(g)w}$, which also descends to $\theta: \fg \rightarrow \fg$. Since $\theta$ is an involution, we have an eigenspace decomposition $\fg= \fg^+\oplus \fg^-$ where $\fg^\pm = \{ X \in \fg \mid \br{Xv,w} \pm\br{v,Xw}=0 \textup{ for any } v, w \in V\}$. Note that $G^\theta$ is isomorphic to the symplectic group $Sp_{2n}$; we identify $W$ in \ref{sec:weyl} with the Weyl group of $G^\theta$. 

Let $\cB$ be the flag variety of $G^\theta$, which we usually realize as
$$\cB=\{F_\bullet = [0=F_0\subset F_1\subset \cdots \subset F_{2n-1} \subset F_{2n}=V] \mid \dim F_i = i, F_i^\perp=F_{2n-i} \textup{ for all } i\in [0,2n]\}.$$
For a nilpotent element $N \in \fg^-$ and $v \in V$, we define $\cB_{N,v}\colonequals \{F_\bullet \in \cB \mid NF_i \subset F_{i} \textup{ for all } i \in [0, 2n] \textup{ and } v \in F_n\}$, called the exotic Springer fiber of $(N, v)$. Then by \cite{kat09} (see also \cite{ss13}), there exists an action of $W$ on $H^i(\cB_{N,v})$ for each $i \in \bZ$, called the exotic Springer representation. When $F(N)=N$ and $F(v)=v$ (where $F$ is the geometric Frobenius with the given $\Fq$-structure), $F^*$ naturally acts on $H^i(\cB_{N,v})$ and it is known that $F^*$ and the $W$-action on $H^i(\cB_{N,v})$ commute.

\section{$G^\theta$-orbits in the exotic nilpotent cone}
Let $\cN^-$ be the set of nilpotent elements in $\fg^-$. The variety $\cN^-\times V$ is called the exotic nilpotent cone, originally defined in \cite[1.1]{kat09}. There is a $G^\theta$-action on $\cN^-\times V$ defined by $g\cdot (X, v) = (\Ad_g(X), g(v))$, i.e. a diagonal action.
\subsection{Parametrization of $G^\theta$-orbits in $\cN^-\times V$} \label{sec:parexo} We first give a parametrization of $G^\theta$-orbits in $\cN^-\times V$ as follows. (ref. \cite{ah08}, \cite{nrs18}) Set $\fP_2$ to be the set of pairs $(\mu,\nu)$ of partitions such that $|\mu|+|\nu|=n$. For a $G^\theta$-orbit $\cO \subset \cN^-\times V$, we attach $(\mu,\nu) \in \fP_2$ as follows. Choose any $(N, v) \in \cO$ and let $\lambda\vdash n$ be the partition such that $\lambda \cup \lambda$ is the Jordan type of $N$, and let $\hat{\lambda}$ be the Jordan types of $N$ on $V/\kk[N]v=V/\spn{N^mv \mid m \in \bN}$. Then $\mu=(\mu_1, \mu_2, \ldots), \nu=(\nu_1, \nu_2, \ldots)$ are set to be the unique pair of partitions such that $\lambda=\mu+\nu$ and $\hat{\lambda}=(\mu_1+\nu_1, \mu_2+\nu_1, \mu_2+\nu_2, \mu_3+\nu_2,\mu_3+\nu_3,\mu_4+\nu_3, \ldots)=(\mu+\nu)\cup(\mu_2+\nu_1, \mu_3+\nu_2, \mu_4+\nu_3, \ldots)$. Here $(\mu, \nu)$ is independent of the choice of $(N, v) \in \cO$ and this gives a bijective correspondence from the set of $G^\theta$-orbits in $\cN^-\times V$ to $\fP_2$.
\begin{rmk} $\mu$ and $\nu$ can be recovered from $\lambda$ and $\hat{\lambda}$ as $\nu_i = \lambda_{i}-\mu_i$ and $\mu_{i+1} = \hat{\lambda}_{2i}-\nu_i$ for $i \in \bZ_{>0}$, and $\mu_1 = 2n-|\hat{\lambda}|$.
\end{rmk}

By \cite[Theorem 1.14]{kat09}, the stabilizer in $G^\theta$ of any element in $\cN^-\times V$ under the diagonal action is connected. Therefore, the Lang-Steinberg theorem implies that for any $G^\theta$-orbit $\cO \subset \cN^-\times V$, its $F$-fixed points $\cO^F$ is nonempty and a single $(G^\theta)^F$-orbit. (See \cite[Chapter 3]{dm91} for more information.)

\subsection{A standard model} \label{sec:standardexo} Here we describe a standard choice of a geometric Frobenius $F$ and a pair $(N, v)$ in each $G^\theta$-orbit in $\cN^-\times V$ parametrized by $(\mu, \nu) \in \fP_2$. (Note that our choice is different from a normal basis in \cite{ah08}.) Set $\lambda=\mu+\nu$. We fix a basis $\{\vv^t_{r,s} \mid r \in \und{\lambda}, s\in [1, 2m_\lambda(r)], t\in [1, r]\}$ of $V$, and define a symplectic form $\br{\ , \ }$ to be
\begin{enumerate}[label=$\bullet$]
\item $\br{\vv^{t}_{r,2k-1},\vv^{r+1-t}_{r,2k}}=1$ for $k \in [1,m_\lambda(r)]$ and $t\in [1,r]$
\item $\br{\vv^{t}_{r,s}, \vv^{t'}_{r',s'}}=0$ otherwise
\end{enumerate}
and extend it to $V$ by bilinearity and skew-symmetry. We also set $F: V \rightarrow V$ to be $F(\sum_{r,s,t}a^t_{r,s} \vv^t_{r,s}) = \sum_{r,s,t} (a^t_{r,s})^q \vv^t_{r,s}$. Then the basis $\{\vv^t_{r,s}\}_{r,s,t}$ and the bilinear form $\br{\ , \ }$ are defined over $\Fq$. 

We define $N \in \End(V)$ as follows: for $r\in \und{\lambda}$ and $s\in[1, 2m_\lambda(r)]$ we set $N\vv_{r,s}^t = \vv_{r,s}^{t+1}$ if $t\in [1, r-1]$ and $N\vv_{r,s}^r = 0$, and extend it to $V$ by linearity. Then it is easy to show that the Jordan type of $N$ is $\lambda\cup \lambda$, $N \in \cN^-$, and $N$ is $F$-stable. We set $v\colonequals \sum_{r_i \in \und{\lambda}_v} \vv_{r_i,1}^{\nu_i+1} \in V$, where
\begin{align*}
\und{\lambda}_v\colonequals \{r_i \in \und{\lambda} \mid &\ \mu_i>\mu_j \textup{ for any } r_j\in\und{\lambda} \textup{ such that } r_i>r_j, \textup{ where } r_i=\mu_i+\nu_i \textup{ and } r_j=\mu_j+\nu_j\}.
\end{align*}Then $v$ is $F$-stable and the pair $(N,v)$ is contained in the $G^\theta$-orbit parametrized by $(\mu, \nu)$.

For $i \in \bN$, we set $\ker_{\geq i}N\colonequals \ker N \cap \im N^{i-1}$ and $\ker_{> i}N\colonequals \ker_{\geq i+1}N$. In particular, we have $\ker_{\geq i}N= \{0\}$ if $i > \lambda_1$. (Here we adopt the convention that $N^0 = Id$.) Then we have a filtration $\ker N = \ker_{\geq 1}N \supset  \ker_{\geq 2}N\supset \cdots$ 
and for $r \in \und{\lambda}$ we have $\ker_{\geq r}N = \spn{\vv^{r'}_{r',s} \mid r' \in \und{\lambda}, r'\geq r, s\in[1, 2m_\lambda(r')]}$. In particular, $\dim \ker_{\geq r}N = 2m_{\lambda}(\geq r)$.

\begin{center}
\begin{figure}[!hbtp]
\begin{tikzpicture}
	\yellowfill{0,-1}{1,0}
	\yellowfill{2,-1}{3,0}
	\yellowfill{8,-1}{9,0}
	\yellowfill{12,-1}{13,0}
	\orangeboxwlabel{0}{3}{2}{7}
	\orangeboxwlabel{2}{3}{2}{6}
	\orangeboxwlabel{4}{3}{2}{5}
	\orangeboxwlabel{6}{2}{2}{4}
	\orangeboxwlabel{8}{1}{4}{3}
	\orangeboxwlabel{12}{0}{2}{1}
	\baseline{14}
\end{tikzpicture}
\caption{$\mu=(4,3,2,2,2,2,1), \nu=(3,3,3,2,1,1)$ case}
\label{fig:exo1}
\end{figure}
\end{center}
\begin{example} \label{ex:exo1} Figure \ref{fig:exo1} illustrates the case when $\mu=(4,3,2,2,2,2,1)$ and $\nu=(3,3,3,2,1,1)$. Here, each box is considered as an element of the basis $\{\vv^{t}_{r,s}\}_{r,s,t}$ and $N$ acts by stepping down one box. Then the Jordan type of $N$ is $\lambda\cup\lambda$ where $\lambda=\mu+\nu=(7,6,5,4,3,3,1)$. Also, there is a dashed line dividing the diagram into two pieces, where the lower part (resp. the upper part) becomes the Young diagram of $\mu\cup \mu$ (resp. $\nu \cup \nu$). To calculate $\hat{\lambda}=(7,6,6,5,5,5,4,4,3,3,3,2,1)$, one may shift the lower part to the left by one box and read the sizes of the columns. Finally, we have $\und{\lambda}_v = \{7,6,3,1\}$ and $v \in V$ is defined to be the sum of basis elements which correspond to the shaded boxes.
\end{example}

\section{Total Springer representations and the restriction formula}
\subsection{Total Springer representations} \label{sec:greenexo} Recall the action of $W$ on $H^i(\cB_{N,v})$ for $(N,v) \in \cN^-\times V$. Suppose that $(N,v)$ is $F$-stable and the $G^\theta$-orbit containing $(N,v)$ is parametrized by $(\mu,\nu)$. Then the function $W \rightarrow \qlbar: w \mapsto \sum_{i\in \bZ}(-1)^i \tr(wF^*, H^i(\cB_{N,v}))$ is a character of $W$ that does not depend on the choice of $(N,v)$, which we denote by $\tspq{\mu,\nu}$.

\begin{rmk} Unlike \ref{sec:green}, in this case we may set $\tspq{\mu,\nu}(w) = \sum_{i \in \bZ} \tr(w, H^{2i}(\cB_{N,v}))q^i$ (after replacing $q$ with its sufficiently large power if necessary). This follows from the existence of an affine paving of $\cB_{N,v}$, see Section \ref{sec:affinepav}.
\end{rmk}

\subsection{Partial Springer resolution and the restriction formula} We claim the following proposition which is the key step of our calculation.
\begin{prop} \label{prop:resexo} Suppose that $(N,v) \in \cN^-\times V$ is an $F$-stable pair in the $G^\theta$-orbit parametrized by $(\mu, \nu) \in \fP_2$. Then we have
$$\Res_{W'}^W \tspq{\mu, \nu} = \sum_{l \in \bP(\ker N)^F} \tspq{\mu(l), \nu(l)}$$
as characters of $W'$. Here, $(\mu(l), \nu(l))$ is the parameter of the $Sp(l^\perp/l)$-orbit containing $(N|_{l^\perp/l}, v+l)$. (Note that $Sp(l^\perp/l)$ is isomorphic to $Sp_{2n-2}$.)
\end{prop}
\begin{proof} To this end, we need a Borho-MacPherson type formula \cite{bm83} concerning partial Springer resolution in the exotic case. Here our argument relies on \cite{ss13}. More precisely, we define $\widetilde{\cX}, \widetilde{\cX}^P, \cX, \pi': \widetilde{\cX} \rightarrow \widetilde{\cX}^P, \pi'': \widetilde{\cX}^P \rightarrow \cX$, and $\pi= \pi''\circ\pi'$ to be as in \cite[3.1, 4.1, 6.4]{ss13}. Then the following two things  need to be checked in order to argue similarly to \cite[Proposition 6.1]{kim:betti}: (1) $\pi'_* \qlbar_{\widetilde{X}}$ is equipped with a $W'$-action so that its induced action on the stalk at $(N, v, gP^\theta)\in \widetilde{\cX}^P$ is isomorphic to the exotic Springer action corresponding to the stabilizer of $l$ in $G^\theta$ (here $g \in G^\theta$ is chosen such that $gP^\theta g^{-1}$ is the stabilizer of $l$ in $G^\theta$), and (2) the induced $W'$-action on $\pi''_*(\pi'_* \qlbar_{\widetilde{X}}) = \pi_* \qlbar_{\widetilde{X}}$ coincides with the restriction of $W$-action to $W'$. Now the first part follows from (6.4.5), (6.5.3), and Theorem 4.2 in \cite{ss13}, and the second part follows from (6.4.5), (6.4.6), and Theorem 4.2 in \cite{ss13}. After this, the rest of the proof is identical with \cite[Proposition 6.1]{kim:betti}.
\end{proof}

\section{Calculation} \label{sec:calcexo}
Let $(N,v) \in \cN^-\times V$ be contained in a $G^\theta$-orbit parametrized by $(\mu, \nu) \in \fP_2$. The goal of this section is to calculate the RHS of the formula in Proposition \ref{prop:resexo}, i.e. for each $r \in \und{\lambda}$ where $\lambda$ is the partition such that $\lambda \cup \lambda$ is the Jordan type of $N$, we calculate $\sum_{l \in (\bP(\ker_{\geq r} N)-\bP(\ker_{> r} N))^F} \tspq{\mu(l), \nu(l)}$ using geometric argument.  For simplicity we assume that $(N,v)$ and $F$ are defined as in \ref{sec:standardexo} and let $l\colonequals \spn{\ww} \subset \ker_{\geq r}N - \ker_{>r}N$ for some $r\in \und{\lambda}$ where $\ww=\sum_{r'\geq r, s\in [1, 2m_{\lambda}(r')]} a_{r',s}\vv^{r'}_{r',s}$. Note that $l$ is $F$-stable if and only if there exists $c \in \kk-\{0\}$ such that $ca_{r',s}\in \Fq$ for all $r' \geq r$ and $s \in [1, 2m_\lambda(r')]$. Now we observe the following lemma.
\begin{lem} \label{lem:lamp} Suppose that $l$ is a line contained in $\bP(\ker_{\geq r} N)-\bP(\ker_{> r} N)$. Then the Jordan type of $N|_{l^\perp/l}$ is $(\lambda\cup \lambda)\pch{r,r}{r-1,r-1}=(\lambda\pch{r}{r-1})\cup (\lambda\pch{r}{r-1})$.
\end{lem}
\begin{proof} It is proved in almost the same manner to \cite[\S 2]{sho83} or \cite[Section 5]{kim:euler}.
\end{proof}

We set $(N', v') \colonequals (N|_{l^\perp/l}, v+l)$, and define $\lambda'$ to be the partition such that $\lambda'\cup \lambda'$ is the Jordan type of $N'$. Then we always have $\lambda' =\lambda\pch{r}{r-1}$ by Lemma \ref{lem:lamp}. Also we set $\hat{\lambda}$ (resp. $\hat{\lambda}'$) to be the Jordan type of $N$ on $V/\kk[N]v$ (resp. $N'$ on $l^\perp/(l+\kk[N]v)\cap l^\perp$). Note that calculating $(\mu', \nu')$ is equivalent to finding $\hat{\lambda}'$ by the remark in \ref{sec:parexo}. To this end, we give an explicit description of a Jordan basis. More precisely, we construct a set 
$$\fB = \{(v_i, r_i) \in (V-\{0\})\times \bZ_{>0} \mid i \in [1, k]\}$$
where $r_i = \min \{m \in \bZ_{>0} \mid N^mv_i \in (l+\kk[N]v)\cap l^\perp\}$ such that $\{N^mv_i+(l+\kk[N]v)\cap l^\perp\mid i\in [1,k], m\in [0, r_i-1]\}$ forms a basis of $l^\perp/(l+\kk[N]v)\cap l^\perp$. In such a case it is clear that $\hat{\lambda}'$ equals $\{r_1, r_2, \ldots, r_k\}$ as a multiset.

 Define $\nabla:\und{\lambda}\cup\{0\} \rightarrow \und{\mu}\cup\{0\}$ and $\Delta:\und{\lambda}\cup\{0\}\rightarrow \und{\nu}\cup\{0\}$  to be functions such that $\nabla(\lambda_{i}) = \mu_i$ and $\Delta(\lambda_{i}) = \nu_i$ for $i \in \bZ_{>0}$. (This is well-defined; if $\lambda_{i}=\lambda_{j}$ then $\mu_i=\mu_j$ and $\nu_i=\nu_j$.) Then the set $\und{\lambda}_v$ in \ref{sec:standardexo} can be written as $\und{\lambda}_v=\{\min \nabla^{-1} (r) \mid r \in \und{\mu}\}=\{r \in \und{\lambda} \mid \nabla(r)>\nabla(r') \textup{ for any } r' \in \und{\lambda}\cup\{0\} \textup{ such that } r'<r\}$, and also $\nabla$ restricts to a bijection $\nabla: \und{\lambda}_v \simeq \und{\mu}$. We set $\ns: \und{\lambda}_v \rightarrow \und{\lambda}_v \cup\{0\}$ (``next step'') to be such that $\ns(r)$ is the biggest value in $\und{\lambda}_v \cup\{0\}$ smaller than $r$. Then we may write $\hat{\lambda} = (\lambda\cup\lambda)\pch{r_1, \ldots, r_k}{\Delta(r_1)+\nabla(\ns(r_1)),\ldots, \Delta(r_k)+\nabla(\ns(r_k))}$ where  $\und{\lambda}_v = \{r_1, r_2, \ldots, r_k\}$. Indeed, this can be shown by choosing the set analogous to $\fB$ above for $V/\kk[N]v$ (instead of $l^\perp/(l+\kk[N]v)\cap l^\perp$) as follows:
\begin{enumerate}[label=$\bullet$, leftmargin=*]
\item $(\vv_{r',s}^1, r')$ if [$r'\not\in \und{\lambda}_v$] or [$r' \in\und{\lambda}_v$ and $s\neq 1$], and 
\item $(\sum_{r''\in \und{\lambda}_v, r''\geq r'}\vv_{r'',1}^{\Delta(r'')-\Delta(r')+1}, \Delta(r')+\nabla(\ns(r')))$ for $r' \in\und{\lambda}_v$.
\end{enumerate}

We denote $\mu(l), \nu(l), m_\lambda, m_{\lambda'}$ by $\mu',\nu', m, m'$, respectively. From now on we divide all the possibilities into the following four cases.
\begin{enumerate}
\item[(\ref{sec:case1exo})] $r \not \in\und{\lambda}_v$.
\item[(\ref{sec:case2exo})] $r \in \und{\lambda}_v$, and $\nabla(r')>\nabla(r), \Delta(r')>\Delta(r)$ for any $r' \in \und{\lambda}$ such that $r'>r$.
\item[(\ref{sec:case3exo})] $r \in \und{\lambda}_v$, and there exists $r' \in \und{\lambda}$ such that $r'>r$ and $\Delta(r')=\Delta(r)$
\item[(\ref{sec:case4exo})] $r \in \und{\lambda}_v$, and there exists $r' \in \und{\lambda}$ such that $r'>r$ and $\nabla(r')=\nabla(r)$
\end{enumerate}

\subsection{}\label{sec:case1exo} First suppose that $r\not \in \und{\lambda}_v$, i.e. $\nabla(r) = \nabla(r')$ for some $r' \in \und{\lambda}$ such that $r'<r$. (It corresponds to $r\in \{4,5\}$ case in Example \ref{ex:exo1}.) 
We claim that $\mu'=\mu$ and $\nu'=\nu\pch{\Delta(r)}{\Delta(r)-1}$. Indeed, recall that $l=\spn{\ww}$ where $\ww=\sum_{r'\geq r, s\in[1, 2m(r')]} a_{r',s}\vv^{r'}_{r',s}$ and suppose that $i \in [1, 2m(r)]$ is the largest value such that $a_{r,i}\neq 0$, in which case it is safe to assume that $a_{r,i}= 1$. If $i$ is odd (resp. even) then we set $i'=i+1$ and $\vv'=-\vv_{r,i+1}^1$ (resp. $i'=i-1$ and $\vv'=\vv_{r,i-1}^1$) so that $\br{\vv', \ww}=1$. Now we choose the elements which comprise the set $\fB$ as follows.
\begin{enumerate}[label=$\bullet$, leftmargin=*]
\item $(\vv_{r',s}^1-\br{\vv_{r',s}^1,\ww}\vv', r')$ for 
\begin{enumerate}[label=$-$]
\item $r'\not\in \und{\lambda}_v\cup\{r\}$ and $s\in [1, 2m(r')]$,
\item $r' \in\und{\lambda}_v$ and $s\in[2,2m(r')]$, or 
\item $r'=r$ and $s\in [1,2m(r)]-\{i, i'\}$
\end{enumerate}
\item $(\vv''-\br{\vv'',\ww}\vv', \Delta(r')+\nabla(\ns(r')))$ where $\vv''=\sum_{r''\in \und{\lambda}_v, r''\geq r'}\vv_{r'',1}^{\Delta(r'')-\Delta(r')+1}$ for $r' \in\und{\lambda}_v$
\item $(\sum_{r'\geq r, s\in[1, 2m(r')]} a_{r',s}\vv^{r'-(r-1)}_{r',s}, r-1)$
\item $(N\vv_{r,i'}^1=\vv^2_{r,i'}, r-1)$
\end{enumerate}
From this description of $\fB$ it follows that $\hat{\lambda}'=\hat{\lambda}\pch{r,r}{r-1,r-1}$ and one can easily calculate that $(\mu', \nu')=(\mu,\nu\pch{\Delta(r)}{\Delta(r)-1})$. Since $\#(\bP(\ker_{\geq r}N)-\bP(\ker_{>r}N))^F=\frac{\qe{2m({\geq r})}-\qe{2m({>r})}}{q-1}$, we have
$$\sum_{l \in (\bP(\ker_{\geq r} N)-\bP(\ker_{> r} N))^F} \tspq{\mu', \nu'}=\frac{\qe{2m({\geq r})}-\qe{2m({>r})}}{q-1} \tspq{\mu,\nu\pch{\Delta(r)}{\Delta(r)-1}}.$$

\subsection{}\label{sec:case2exo} Let us assume that $r\in \und{\lambda}_v$ and $\nabla(r')>\nabla(r), \Delta(r')>\Delta(r)$ for any $r' \in \und{\lambda}$ such that $r'>r$. (It corresponds to $r\in \{1,7\}$ case in Example \ref{ex:exo1}. In particular, this includes the case when $r=\lambda_1$.) 
Recall that $l=\spn{\ww}$ where $\ww=\sum_{r'\geq r, s\in [1, 2m_{\lambda}(r')]} a_{r',s}\vv^{r'}_{r',s}$.
Let $H\subset \bP(\ker_{\geq r}N)$ be the hyperplane defined by $a_{r, 2}=0$ which contains $\bP(\ker_{>r}N)$. If $l \not \in H$ then we may assume that $a_{r,2}=1$. Assume that $\Delta(r) >0$ for now and set $\vv'=\sum_{r'\in \und{\lambda}_v, r'\geq \sigma}\vv_{r',1}^{\Delta(r')-\Delta(r)+1}$, where $\sigma=\min\{r' \in \und{\lambda}_v  \mid \Delta(r) = \Delta(r')\}$. The the assumption implies that $\br{\vv',\ww}=1$. Here we take the set $\fB$ as follows.
\begin{enumerate}[label=$\bullet$, leftmargin=*]
\item $(\vv_{r',s}^1-\br{\vv_{r',s}^1,\ww}\vv', r')$ for
\begin{enumerate}[label=$-$]
\item $r'\not\in \und{\lambda}_v$ and $s\in [1, 2m(r')]$,
\item $r' \in\und{\lambda}_v-\{r\}$ and $s\in[2,2m(r')]$, or
\item $r'=r$ and $s\in [3,2m(r)]$
\end{enumerate}
\item $(\vv''-\br{\vv'',\ww}\vv', \Delta(r')+\nabla(\ns(r')))$ where $\vv''=\sum_{r''\in \und{\lambda}_v, r''\geq r'}\vv_{r'',1}^{\Delta(r'')-\Delta(r')+1}$ for $r' \in\und{\lambda}_v-\{\sigma\}$
\item $(\sum_{r'\geq r, s\in [1, 2m(r')]} a_{r',s}\vv^{r'-(r-1)}_{r',s}, r-1)$
\item $(N\vv'=\sum_{r'\in \und{\lambda}_v, r'\geq \sigma}\vv_{r',1}^{\Delta(r')-\Delta(r)+2}, \Delta(\sigma)+\nabla(\ns(\sigma))-1)$
\end{enumerate}
From this we have $\hat{\lambda}' = \hat{\lambda}\pch{r, \Delta(\sigma)+\nabla(\ns(\sigma))}{r-1,\Delta(\sigma)+\nabla(\ns(\sigma))-1}$, and one can show that $\mu'=\mu^{+[m(\geq r)+1,M(\nu)]}$ and $\nu'=(\nu\pch{\Delta(r)}{\Delta(r-1)})^{-[m(\geq r)+1,M(\nu)]}=\nu^{-[m(\geq r),M(\nu)]}$ where $M(\nu)=m_\nu(\geq\Delta(r))$.
As $\#(\bP(\ker_{\geq r}N)-H)^F = \qe{2m(\geq r)-1}$, we have
$$\sum_{l \in (\bP(\ker_{\geq r}N)-H)^F}\tspq{\mu',\nu'} =  \qe{2m(\geq r)-1}\tspq{\mu^{+[m(\geq r)+1,M(\nu)]},\nu^{-[m(\geq r),M(\nu)]}}.$$
\begin{rmk} If $\Delta(r)=0$ (as $r=1$ case in Example \ref{ex:exo1}), then $\nu^{-[m(\geq r),m_\nu(\geq \Delta(r))]}$ is not well-defined. In this case we simply ignore this term, since in this case $v\not\in l^\perp$ which means that the corresponding exotic Springer fiber is an empty set.
\end{rmk}

Set $H' \subset H$ to be a linear subvariety defined by ($a_{r,2}=0$ and) $a_{r,3}=\cdots=a_{r,2m(r)}=0$ which contains $\bP(\ker_{>r}N)$. We assume $m(r)>1$, $l\in H-H'$, and that $i \in [3, 2m(r)]$ is the largest integer such that $a_{r,i}\neq 0$, in which case it is safe to assume that $a_{r,i}= 1$. If $i$ is odd (resp. even) then we set $i'=i+1$ and $\vv'=-\vv_{r,i+1}^1$ (resp. $i'=i-1$ and $\vv'=\vv_{r,i-1}^1$) so that $\br{\vv', \ww}=1$. Now we take the set $\fB$ as follows.
\begin{enumerate}[label=$\bullet$, leftmargin=*]
\item $(\vv_{r',s}^1-\br{\vv_{r',s}^1,\ww}\vv', r')$ for 
\begin{enumerate}[label=$-$]
\item $r'\not\in \und{\lambda}_v$ and $s\in [1, 2m(r')]$,
\item $r' \in\und{\lambda}_v-\{r\}$ and $s\in[2,2m(r')]$, or 
\item $r'=r$ and $s\in [2,2m(r)]-\{i, i'\}$
\end{enumerate}
\item $(\vv''-\br{\vv'',\ww}\vv', \Delta(r')+\nabla(\ns(r')))$ where $\vv''=\sum_{r''\in \und{\lambda}_v, r''\geq r'}\vv_{r'',1}^{\Delta(r'')-\Delta(r')+1}$ for $r' \in\und{\lambda}_v$
\item  $(\sum_{r'\geq r, s\in[1,2m(r')]} a_{r',s}\vv^{r'-(r-1)}_{r',s}, r-1)$
\item $(N\vv^1_{r,i'}=\vv^2_{r,i'}, r-1)$
\end{enumerate}
From this it follows that $\hat{\lambda}' = \hat{\lambda}\pch{r,r}{r-1,r-1}$, which implies that $\mu'=\mu\pch{\nabla(r)}{\nabla(r)-1}$ and $\nu'=\nu$. As $\#(H-H')^F = \frac{\qe{2m(\geq r)-1}-\qe{2m(>r)+1}}{q-1}$, we have
$$\sum_{l \in (H-H')^F}\tspq{\mu',\nu'} =   \frac{\qe{2m(\geq r)-1}-\qe{2m(>r)+1}}{q-1}\tspq{\mu\pch{\nabla(r)}{\nabla(r)-1},\nu}.$$
When $m(r)=1$, then $H=H'$, i.e. $H-H'=\emptyset$, which is consistent with the fact that the coefficient of the term above is zero. In such a case we simply ignore this term.

Finally we assume that $l \in H'$. Then we may assume that $a_{r,1}=1$ and thus $\br{-\vv_{r,2}^1,\ww}=1$. First suppose that $r=\lambda_1$ so that $\ww=\vv_{r,1}^r$. Then we take the set $\fB$ as follows. 
\begin{enumerate}[label=$\bullet$, leftmargin=*]
\item $(\vv_{r',s}^1+\br{\vv_{r',s}^1,\ww}\vv_{r,2}^1, r')$ for 
\begin{enumerate}[label=$-$]
\item $r'\not\in \und{\lambda}_v$ and $s\in [1, 2m(r')]$,
\item $r' \in\und{\lambda}_v-\{r\}$ and $s\in[2,2m(r')]$, or 
\item $r'=r$ and $s\in [3,2m(r)]$
\end{enumerate}
\item $(\vv'', \Delta(r')+\nabla(\ns(r')))$ where $\vv''=\sum_{r''\in \und{\lambda}_v, r''\geq r'}\vv_{r'',1}^{\Delta(r'')-\Delta(r')+1}$ for $r' \in\und{\lambda}_v$
\item $(N\vv^1_{r,2}=\vv^2_{r,2}, r-1)$
\end{enumerate}
(Note that $l \subset \kk[N]v$ in this case.) It follows that $\hat{\lambda}' = \hat{\lambda}\pch{r}{r-1}$. On the other hand, if we assume that $r \neq \lambda_1$ then there exists $\varpi \in \und{\lambda}_v$ such that $\ns(\varpi)=r$. Now we take the set $\fB$ as follows. 
\begin{enumerate}[label=$\bullet$, leftmargin=*]
\item $(\vv_{r',s}^1+\br{\vv_{r',s}^1,\ww}\vv_{r,2}^1, r')$ for 
\begin{enumerate}[label=$-$]
\item $r'\not\in \und{\lambda}_v$ and $s\in [1, 2m(r')]$,
\item $r' \in\und{\lambda}_v-\{r\}$ and $s\in[2,2m(r')]$, or 
\item $r'=r$ and $s\in [3,2m(r)]$
\end{enumerate}
\item $(\vv''+\br{\vv'',\ww}\vv_{r,2}^1, \Delta(r')+\nabla(\ns(r')))$ where $\vv''=\sum_{r''\in \und{\lambda}_v, r''\geq r'}\vv_{r'',1}^{\Delta(r'')-\Delta(r')+1}$ for $r' \in\und{\lambda}_v-\{\varpi\}$
\item $(\vv'''+\br{\vv''',\ww}\vv_{r,2}^1, \Delta(\varpi)+\nabla(r)-1)$ where 
$$\vv''' = \sum_{r'\in \und{\lambda}_v, r'\geq \varpi}\vv_{r',1}^{\Delta(r')-\Delta(\varpi)+1}-\sum_{r'>r, s\in[1,2m(r')]} a_{r',s}\vv^{r'-(\Delta(\varpi)+\nabla(r)-1)}_{r',s}$$
\item $(N\vv^1_{r,2}=\vv^2_{r,2}, r-1)$
\end{enumerate}
It follows that $\hat{\lambda}' = \hat{\lambda}\pch{r, \Delta(\varpi)+\nabla(\ns(\varpi))}{r-1, \Delta(\varpi)+\nabla(\ns(\varpi))-1}$ as $\ns(\varpi) = r$. In any case, direct calculation shows that $\mu'=\mu^{-[m(>r)+1, M(\mu)]}$ and $\nu'=(\nu^{+[m(>r)+1,M(\mu)]})\pch{\Delta(r)+1}{\Delta(r)} = \nu^{+[m(>r)+1,M(\mu)-1]}$.
where $M(\mu)=m_\mu(\geq \nabla( r))$.
As $\#(H'-\bP(\ker_{>r}))^F = \qe{2m(>r)}$, it follows that
$$\sum_{l \in (H'-\bP(\ker_{>r}N))^F}\tspq{\mu',\nu'} =  \qe{2m(> r)}\tspq{\mu^{-[m(>r)+1, M(\mu)]},\nu^{+[m(>r)+1,M(\mu)-1]}}.$$

\subsection{}\label{sec:case3exo}
This time we assume that $r \in \und{\lambda}_v$ and there exists $r' \in \und{\lambda}$ such that $r'>r$ and $\Delta(r')=\Delta(r)$. (This corresponds to $r=6$ case in Example \ref{ex:exo1}.) We set $\sigma$ and $\beta$ to be the smallest and biggest entry in $\und{\lambda}_v $, respectively, such that $\Delta(\sigma) = \Delta(\beta) = \Delta(r)$.
%
%
We define $H\subset \bP(\ker_{\geq r}N)$ to be the hyperplane defined by $\sum_{r' \in \und{\lambda}_v, r\leq r'\leq	\beta} a_{r',2}=0$. Here $H$ does not contain $\bP(\ker_{>r}N)$ but $H$ and $\bP(\ker_{>r}N)$ intersect transversally in $\bP(\ker_{\geq r}N)$. We assume that $l \not\in H$ and thus we may set $\sum_{r' \in \und{\lambda}_v, r\leq r'\leq \beta} a_{r',2}=1$. Suppose that $\Delta(r)>0$ for now and let $\vv' =\sum_{r'\in \und{\lambda}_v, r'\geq \sigma}\vv_{r',1}^{\Delta(r')-\Delta(r)+1}$, so that $\br{\vv',\ww}=1$. Also let $i \in [1, 2m(r)]$ be the biggest value such that $a_{r,i}\neq 0$. If $i \neq 1$, then we take the set $\fB$ as follows.
\begin{enumerate}[label=$\bullet$, leftmargin=*]
\item $(\vv_{r',s}^1-\br{\vv_{r',s}^1,\ww}\vv', r')$ for 
\begin{enumerate}[label=$-$]
\item $r'\not\in \und{\lambda}_v$ and $s\in [1, 2m(r')]$,
\item $r' \in\und{\lambda}_v-\{r\}$ and $s\in[2,2m(r')]$, or 
\item $r'=r$ and $s\in [2,2m(r)]-\{i\}$
\end{enumerate}
\item $(\vv''-\br{\vv'',\ww}\vv', \Delta(r')+\nabla(\ns(r')))$ where $\vv''=\sum_{r''\in \und{\lambda}_v, r''\geq r'}\vv_{r'',1}^{\Delta(r'')-\Delta(r')+1}$ for $r' \in\und{\lambda}_v-\{\sigma\}$
\item $(\sum_{r'\geq r, s\in[1, 2m(r')]} a_{r',s}\vv^{r'-(r-1)}_{r',s}, r-1)$
\item $(N\vv' = \sum_{r'\in \und{\lambda}_v, r'\geq \sigma}\vv_{r',1}^{\Delta(r')-\Delta(r)+2}, \Delta(r)+\nabla(\ns(\sigma))-1)$
\end{enumerate}
It follows that $\hat{\lambda}' = \hat{\lambda}\pch{r, \Delta(\sigma)+\nabla(\ns(\sigma))}{r-1, \Delta(\sigma)+\nabla(\ns(\sigma))-1}$ as $\Delta(r) = \Delta(\sigma)$.
On the other hand, if $i = 1$ then set $\varpi$ to be the entry in $\und{\lambda}_v$ such that $\ns(\varpi) = r$ and we take the set $\fB$ as follows.
\begin{enumerate}[label=$\bullet$, leftmargin=*]
\item $(\vv_{r',s}^1-\br{\vv_{r',s}^1,\ww}\vv', r')$ for 
\begin{enumerate}[label=$-$]
\item $r'\not\in \und{\lambda}_v$ and $s\in [1, 2m(r')]$, or
\item $r' \in\und{\lambda}_v$ and $s\in[2,2m(r')]$
\end{enumerate}
\item $(\vv''-\br{\vv'',\ww}\vv', \Delta(r')+\nabla(\ns(r')))$ where $\vv''=\sum_{r''\in \und{\lambda}_v, r''\geq r'}\vv_{r'',1}^{\Delta(r'')-\Delta(r')+1}$ for $r' \in\und{\lambda}_v-\{\sigma, \varpi\}$
\item $(\vv'''-\br{\vv''',\ww}\vv', \Delta(\varpi)+\nabla(r)-1)$ where 
$$\vv''' = \sum_{r'\in \und{\lambda}_v, r'\geq \varpi}\vv_{r',1}^{\Delta(r')-\Delta(\varpi)+1}-\sum_{r'> r, s\in[1,2m(r')]} a_{r',s}\vv^{r'-(\Delta(\varpi)+\nabla(r)-1)}_{r',s}$$
\item $(N\vv' = \sum_{r'\in \und{\lambda}_v, r'\geq \sigma}\vv_{r',1}^{\Delta(r')-\Delta(r)+2}, \Delta(r)+\nabla(\ns(\sigma))-1)$
\end{enumerate}
It follows that $\hat{\lambda}' = \hat{\lambda}\pch{r, \Delta(\sigma)+\nabla(\ns(\sigma))}{r-1, \Delta(\sigma)+\nabla(\ns(\sigma))-1}$ as $\Delta(\varpi) = \Delta(r)= \Delta(\sigma)$. Thus in any case, we have $\mu'=\mu^{+[m(\geq r)+1,M(\nu)]}$ and $\nu'=(\nu\pch{\Delta(r)}{\Delta(r-1)})^{-[m(\geq r)+1,M(\nu)]}=\nu^{-[m(\geq r),M(\nu)]}$
where $M(\nu)=m_\nu(\geq \Delta(r))$.
As $\#(\bP(\ker_{\geq r}N) - H\cup\bP(\ker_{>r}N))^F = \qe{2m(\geq r)-1}-\qe{2m(>r)-1}$, we have
\begin{align*}
&\sum_{l \in (\bP(\ker_{\geq r}N) - H\cup\bP(\ker_{>r}N))^F}\tspq{\mu',\nu'} 
\\&= (\qe{2m(\geq r)-1}-\qe{2m(>r)-1})\tspq{\mu^{+[m(\geq r)+1,M(\nu)]},\nu^{-[m(\geq r),M(\nu)]}}.
\end{align*}
\begin{rmk} As in \ref{sec:case2exo}, if $\Delta(r)=0$  then $\nu^{-[m(\geq r),m_\nu(\geq \Delta(r))]}$ is not well-defined. In this case we simply ignore this term, since in this case $v\not\in l^\perp$ which means that the corresponding exotic Springer fiber is an empty set.
\end{rmk}

Now we assume that $l \in H$ and let $i \in [1, 2m(r)]$ be the biggest value such that $a_{r,i}\neq 0$, in which case it is safe to set $a_{r,i}= 1$. First we assume that $i \geq 3$. If $i$ is odd (resp. even) then we set $i'=i+1$ and $\vv'=-\vv_{r,i+1}^1$ (resp. $i'=i-1$ and $\vv'=\vv_{r,i-1}^1$) so that $\br{\vv', \ww}=1$. If $i \not\in \{1,2\}$, then we take the set $\fB$ as follows.
\begin{enumerate}[label=$\bullet$, leftmargin=*]
\item $(\vv_{r',s}^1-\br{\vv_{r',s}^1,\ww}\vv', r')$ for 
\begin{enumerate}[label=$-$]
\item $r'\not\in \und{\lambda}_v$ and $s\in [1, 2m(r')]$,
\item $r' \in\und{\lambda}_v-\{r\}$ and $s\in[2,2m(r')]$, or 
\item $r'=r$ and $s\in [2,2m(r)]-\{i, i'\}$
\end{enumerate}
\item $(\vv''-\br{\vv'',\ww}\vv', \Delta(r')+\nabla(\ns(r')))$ where $\vv''=\sum_{r''\in \und{\lambda}_v, r''\geq r'}\vv_{r'',1}^{\Delta(r'')-\Delta(r')+1}$ for $r' \in\und{\lambda}_v$
\item  $(\sum_{r'\geq r, s\in[1, 2m(r')]} a_{r',s}\vv^{r'-(r-1)}_{r',s}, r-1)$
\item $(N\vv^1_{r,i'}=\vv^2_{r,i'}, r-1)$
\end{enumerate}
If $i=2$ then set $\varpi \in \und{\lambda}_v$ to be such that $\ns(\varpi)=r$ and we take the set $\fB$ as follows. (Note that $\br{\vv_{r,1}^1,\ww}=1$.)
\begin{enumerate}[label=$\bullet$, leftmargin=*]
\item $(\vv_{r',s}^1-\br{\vv_{r',s}^1,\ww}\vv_{r,1}^1, r')$ for 
\begin{enumerate}[label=$-$]
\item $r'\not\in \und{\lambda}_v$ and $s\in [1, 2m(r')]$,
\item $r' \in\und{\lambda}_v-\{r\}$ and $s\in[2,2m(r')]$, or 
\item $r'=r$ and $s\in [3,2m(r)]$
\end{enumerate}
\item $(\vv''-\br{\vv'',\ww}\vv_{r,1}^1, \Delta(r')+\nabla(\ns(r')))$ where $\vv''=\sum_{r''\in \und{\lambda}_v, r''\geq r'}\vv_{r'',1}^{\Delta(r'')-\Delta(r')+1}$ for $r' \in\und{\lambda}_v-\{\varpi\}$
\item  $(\sum_{r'\geq r, s\in[1, 2m(r')]} a_{r',s}\vv^{r'-(r-1)}_{r',s}, r-1)$
\item $(N\vv_{r,1}^1=\vv_{r,1}^2, r-1)$
\end{enumerate}
(Here we use the fact that $\br{\sum_{r'\in \und{\lambda}_v, r'\geq \varpi}\vv_{r',1}^{\Delta(r')-\Delta(\varpi)+1},\ww}=-a_{r,2}=-1$.) Finally, if $i=1$ then again set $\varpi \in \und{\lambda}_v$ to be such that $\ns(\varpi)=r$ and  take the set $\fB$ as follows. (Note that $\br{-\vv_{r,2}^1,\ww}=1$.)
\begin{enumerate}[label=$\bullet$, leftmargin=*]
\item $(\vv_{r',s}^1+\br{\vv_{r',s}^1,\ww}\vv_{r,2}^1, r')$ for 
\begin{enumerate}[label=$-$]
\item $r'\not\in \und{\lambda}_v$ and $s\in [1, 2m(r')]$,
\item $r' \in\und{\lambda}_v-\{r\}$ and $s\in[2,2m(r')]$, or 
\item $r'=r$ and $s\in [3,2m(r)]$
\end{enumerate}
\item $(\vv''+\br{\vv'',\ww}\vv_{r,2}^1, \Delta(r')+\nabla(\ns(r')))$ where $\vv''=\sum_{r''\in \und{\lambda}_v, r''\geq r'}\vv_{r'',1}^{\Delta(r'')-\Delta(r')+1}$ for $r' \in\und{\lambda}_v-\{\varpi\}$
\item $(\vv''',\Delta(\varpi)+\nabla(r)-1)$ where 
$$\vv''' = \sum_{r'\in \und{\lambda}_v, r'\geq \varpi}\vv_{r',1}^{\Delta(r')-\Delta(\varpi)+1}-\sum_{r'\geq \varpi, s\in[1,2m(r')]} a_{r',s}\vv^{r'-(\Delta(\varpi)+\nabla(r)-1)}_{r',s}$$
\item $(N\vv_{r,2}^1=\vv_{r,2}^2, r-1)$
\end{enumerate}
In any case, it follows that $\hat{\lambda}' = \hat{\lambda}\pch{r,r}{r-1,r-1}$, which implies that $\mu'=\mu\pch{\nabla(r)}{\nabla(r)-1}$ and $\nu'=\nu$. As $\#(H-\bP(\ker_{>r}N))^F = \frac{\qe{2m(\geq r)-1}-\qe{2m(>r)-1}}{q-1}$, we have
$$\sum_{l \in (H-\bP(\ker_{>r}N))^F}\tspq{\mu',\nu'} =   \frac{\qe{2m(\geq r)-1}-\qe{2m(>r)-1}}{q-1}\tspq{\mu\pch{\nabla(r)}{\nabla(r)-1},\nu}.$$

\subsection{}\label{sec:case4exo} Finally, we assume that $r$ is contained in $\und{\lambda}_v$ and there exists $r' \in \und{\lambda}$ such that $r'>r$ and $\nabla(r')=\nabla(r)$. (It corresponds to $r=3$ case in Example \ref{ex:exo1}.)
Let $H \subset \bP(\ker_{\geq r}N)$ be the hyperplane defined by $a_{r,2}= 0$ which contains $\bP(\ker_{>r}N)$. If $l \not\in H$, then we may assume that $a_{r,2}=1$. We assume $\Delta(r)>0$ for now and set $\vv'=\sum_{r'\in \und{\lambda}_v, r'\geq \sigma}\vv_{r',1}^{\Delta(r')-\Delta(r)+1}$, where $\sigma=\min\{r' \in \und{\lambda}_v   \mid \Delta(r) = \Delta(r')\}$.  The the assumption implies that $\br{\vv',\ww}=1$. Here we take the set $\fB$ as follows.
\begin{enumerate}[label=$\bullet$, leftmargin=*]
\item $(\vv_{r',s}^1-\br{\vv_{r',s}^1,\ww}\vv', r')$ for
\begin{enumerate}[label=$-$]
\item $r'\not\in \und{\lambda}_v$ and $s\in [1, 2m(r')]$,
\item $r' \in\und{\lambda}_v-\{r\}$ and $s\in[2,2m(r')]$, or
\item $r'=r$ and $s\in [3,2m(r)]$
\end{enumerate}
\item $(\vv''-\br{\vv'',\ww}\vv', \Delta(r')+\nabla(\ns(r')))$ where $\vv''=\sum_{r''\in \und{\lambda}_v, r''\geq r'}\vv_{r'',1}^{\Delta(r'')-\Delta(r')+1}$ for $r' \in\und{\lambda}_v-\{\sigma\}$
\item $(\sum_{r'\geq r, s\in[1, 2m(r')]} a_{r',s}\vv^{r'-(r-1)}_{r',s}, r-1)$
\item $(N\vv'=\sum_{r'\in \und{\lambda}_v, r'\geq \sigma}\vv_{r',1}^{\Delta(r')-\Delta(r)+2}, \Delta(\sigma)+\nabla(\ns(\sigma))-1)$
\end{enumerate}
From this we have $\hat{\lambda}' = \hat{\lambda}\pch{r, \Delta(\sigma)+\nabla(\ns(\sigma))}{r-1,\Delta(\sigma)+\nabla(\ns(\sigma))-1}$, from which it follows that $\mu'=\mu^{+[m(\geq r)+1,M(\nu)]}$ and $\nu'=(\nu\pch{\Delta(r)}{\Delta(r-1)})^{-[m(\geq r)+1,M(\nu)]}=\nu^{-[m(\geq r),M(\nu)]}$
where $M(\nu)=m_\nu(\geq \Delta(r))$.
As $\#(\bP(\ker_{\geq r}N)-H)^F = \qe{2m(\geq r)-1}$, we have
$$\sum_{l \in (\bP(\ker_{\geq r}N)-H)^F}\tspq{\mu',\nu'} =  \qe{2m(\geq r)-1}\tspq{\mu^{+[m(\geq r)+1,M(\nu)]},\nu^{-[m(\geq r),M(\nu)]}}.$$
\begin{rmk} As in \ref{sec:case2exo} and  \ref{sec:case3exo}, if $\Delta(r)=0$  then $\nu^{-[m(\geq r),m_\nu(\geq \Delta(r))]}$ is not well-defined. In this case we simply ignore this term, since in this case $v\not\in l^\perp$ which means that the corresponding exotic Springer fiber is an empty set.
\end{rmk}

Set $H' \subset H$ to be a linear subvariety defined by ($a_{r,2}=0$ and) $a_{r,3}=\cdots=a_{r,2m(r)}=0$ which contains $\bP(\ker_{>r}N)$. We assume $m(r)>1$, $l\in H-H'$, and that $i \in [3, 2m(r)]$ is the biggest value such that $a_{r,i}\neq 0$, in which case it is safe to assume that $a_{r,i}= 1$. Now if $i$ is odd (resp. even) then we set $i'=i+1$ and $\vv'=-\vv_{r,i+1}^1$ (resp. $i'=i-1$ and $\vv'=\vv_{r,i-1}^1$) so that $\br{\vv', \ww}=1$. Then we take the set $\fB$ as follows.
\begin{enumerate}[label=$\bullet$, leftmargin=*]
\item $(\vv_{r',s}^1-\br{\vv_{r',s}^1,\ww}\vv', r')$ for 
\begin{enumerate}[label=$-$]
\item $r'\not\in \und{\lambda}_v$ and $s\in [1, 2m(r')]$,
\item $r' \in\und{\lambda}_v-\{r\}$ and $s\in[2,2m(r')]$, or 
\item $r'=r$ and $s\in [2,2m(r)]-\{i, i'\}$
\end{enumerate}
\item $(\vv''-\br{\vv'',\ww}\vv', \Delta(r')+\nabla(\ns(r')))$ where $\vv''=\sum_{r''\in \und{\lambda}_v, r''\geq r'}\vv_{r'',1}^{\Delta(r'')-\Delta(r')+1}$ for $r' \in\und{\lambda}_v$
\item  $(\sum_{r'\geq r, s\in[1,2m(r')]} a_{r',s}\vv^{r'-(r-1)}_{r',s}, r-1)$
\item $(N\vv^1_{r,i'}=\vv^2_{r,i'}, r-1)$
\end{enumerate}
From this it follows that $\hat{\lambda}' = \hat{\lambda}\pch{r,r}{r-1,r-1}$, which implies that $\mu'=\mu\pch{\nabla(r)}{\nabla(r)-1}$ and $\nu'=\nu$. As $\#(H-H')^F = \frac{\qe{2m(\geq r)-1}-\qe{2m(>r)+1}}{q-1}$, we have
$$\sum_{l \in (H-H')^F}\tspq{\mu',\nu'} =   \frac{\qe{2m(\geq r)-1}-\qe{2m(>r)+1}}{q-1}\tspq{\mu\pch{\nabla(r)}{\nabla(r)-1},\nu}.$$
When $m(r)=1$, then $H=H'$, i.e. $H-H'=\emptyset$ which is consistent with the fact that the coefficient of the term above is zero. In such a case we simply ignore this term.

Now we assume that $l \in H'$, in which case we may set $a_{r,1}=1$. Let $j \in \und{\lambda}$ be the largest element such that $\nabla(j)=\nabla(r)$ and $\Delta(j)<\Delta(r')$ for any $r'\in \und{\lambda}$ such that $r'>j$. Also we set $\und{\lambda}\cap[r,j] = \{j_1, j_2, \ldots, j_a,j_{a+1}=r\}$ where $j=j_1>j_2>\cdots>j_a>j_{a+1}=r$. Then we have $\nabla(j_1) = \nabla(j_2)=\cdots=\nabla(j_a) = \nabla(r)$. Note that the map $l \mapsto \ww-\vv_{r,1}$ gives an isomorphism of varieties $H' \simeq \ker_{>r}N$. First suppose that $\ww-\vv_{r,1}^r \in \ker_{>j}N$. If $r=\max \und{\lambda}_v$ then $\ww=\vv_{r,1}^r$ and we take the set $\fB$ as follows. 
\begin{enumerate}[label=$\bullet$, leftmargin=*]
\item $(\vv_{r',s}^1+\br{\vv_{r',s}^1,\ww}\vv_{r,2}^1, r')$ for 
\begin{enumerate}[label=$-$]
\item $r'\not\in \und{\lambda}_v$ and $s\in [1, 2m(r')]$,
\item $r' \in\und{\lambda}_v-\{r\}$ and $s\in[2,2m(r')]$, or 
\item $r'=r$ and $s\in [3,2m(r)]$
\end{enumerate}
\item $(\vv''+\br{\vv'',\ww}\vv_{r,2}^1, \Delta(r')+\nabla(\ns(r')))$ where $\vv''=\sum_{r''\in \und{\lambda}_v, r''\geq r'}\vv_{r'',1}^{\Delta(r'')-\Delta(r')+1}$ for $r' \in\und{\lambda}_v$
\item $(N\vv^1_{r,2}=\vv^2_{r,2}, r-1)$
\end{enumerate}
(Note that $l \subset \kk[N]v$ in this case.) It follows that $\hat{\lambda}' = \hat{\lambda}\pch{r}{r-1}$. On the other hand, if we assume that $r \neq \max \und{\lambda}_v$ then there exists $\varpi \in \und{\lambda}_v$ such that $\ns(\varpi)=r$. Now we take the set $\fB$ as follows. 
\begin{enumerate}[label=$\bullet$, leftmargin=*]
\item $(\vv_{r',s}^1+\br{\vv_{r',s}^1,\ww}\vv_{r,2}^1, r')$ for 
\begin{enumerate}[label=$-$]
\item $r'\not\in \und{\lambda}_v$ and $s\in [1, 2m(r')]$,
\item $r' \in\und{\lambda}_v-\{r\}$ and $s\in[2,2m(r')]$, or 
\item $r'=r$ and $s\in [3,2m(r)]$
\end{enumerate}
\item $(\vv''+\br{\vv'',\ww}\vv_{r,2}^1, \Delta(r')+\nabla(\ns(r')))$ where $\vv''=\sum_{r''\in \und{\lambda}_v, r''\geq r'}\vv_{r'',1}^{\Delta(r'')-\Delta(r')+1}$ for $r' \in\und{\lambda}_v-\{\varpi\}$
\item $(\vv'''+\br{\vv''',\ww}\vv_{r,2}^1, \Delta(\varpi)+\nabla(r)-1)$ where 
$$\vv''' = \sum_{r'\in \und{\lambda}_v, r'\geq \varpi}\vv_{r',1}^{\Delta(r')-\Delta(\varpi)+1}-\sum_{r'> r, s\in[1, 2m(r')]} a_{r',s}\vv^{r'-(\Delta(\varpi)+\nabla(r)-1)}_{r',s}$$
\item $(N\vv^1_{r,2}=\vv^2_{r,2}, r-1)$
\end{enumerate}
It follows that $\hat{\lambda}' = \hat{\lambda}\pch{r, \Delta(\varpi)+\nabla(\ns(\varpi))}{r-1, \Delta(\varpi)+\nabla(\ns(\varpi))-1}$ as $\ns(\varpi) = r$, which implies that $\mu'=\mu^{-[m(>j)+1, M(\mu)]}$ and $\nu'=(\nu^{+[m(>j)+1,M(\mu)]})\pch{\Delta(r)+1}{\Delta(r)} = \nu^{+[m(>j)+1,M(\nu)-1]}$
where $M(\mu)=m_\mu(\mathord{\geq\nabla(r)})$.
As $\#(\ker_{>j}N)^F=\qe{2m(>j)}$, it follows that
$$\sum_{\ww-\vv_{r,1}^r \in (\ker_{>j}N)^F}\tspq{\mu',\nu'} =  \qe{2m(>j)}\tspq{\mu^{-[m(>j)+1, M(\mu)]},\nu^{+[m(>j)+1,M(\mu)-1]}}.$$

This time, suppose that $\ww-\vv^r_{r,1} \in \ker_{> j_{b+1}}N-\ker_{>j_b}N$ for some $b \in [1,a]$. Let $i \in [1, 2m(j_b)]$ be the largest element satisfying $a_{j_b,i}\neq 0$. If $r=\max \und{\lambda}_v$ then we take the set $\fB$ as follows. 
\begin{enumerate}[label=$\bullet$, leftmargin=*]
\item $(\vv_{r',s}^1+\br{\vv_{r',s}^1,\ww}\vv_{r,2}^1, r')$ for 
\begin{enumerate}[label=$-$]
\item $r'\not\in \und{\lambda}_v\cup\{j_b\}$ and $s\in [1, 2m(r')]$,
\item $r' \in\und{\lambda}_v-\{r\}$ and $s\in[2,2m(r')]$, 
\item $r' =j_b$ and $s\in[1,2m(r')]-\{i\}$, or 
\item $r'=r$ and $s\in [3,2m(r)]$
\end{enumerate}
\item $(\vv''+\br{\vv'',\ww}\vv_{r,2}^1, \Delta(r')+\nabla(\ns(r')))$ where $\vv''=\sum_{r''\in \und{\lambda}_v, r''\geq r'}\vv_{r'',1}^{\Delta(r'')-\Delta(r')+1}$ for $r' \in\und{\lambda}_v$
\item $(\sum_{r'\geq j_b, s\in[1, 2m(r')]} a_{r',s}\vv^{r'-j_b+1}_{r',s}, j_b-1)$
\item $(N\vv^1_{r,2}=\vv^2_{r,2}, r-1)$
\end{enumerate}
On the other hand, if $r \neq \max \und{\lambda}_v$ then we take the set $\fB$ as follows. 
\begin{enumerate}[label=$\bullet$, leftmargin=*]
\item $(\vv_{r',s}^1+\br{\vv_{r',s}^1,\ww}\vv_{r,2}^1, r')$ for 
\begin{enumerate}[label=$-$]
\item $r'\not\in \und{\lambda}_v\cup\{j_b\}$ and $s\in [1, 2m(r')]$,
\item $r' \in\und{\lambda}_v-\{r\}$ and $s\in[2,2m(r')]$, 
\item $r' =j_b$ and $s\in[1,2m(r')]-\{i\}$, or 
\item $r'=r$ and $s\in [3,2m(r)]$
\end{enumerate}
\item $(\vv''+\br{\vv'',\ww}\vv_{r,2}^1, \Delta(r')+\nabla(\ns(r')))$ where $\vv''=\sum_{r''\in \und{\lambda}_v, r''\geq r'}\vv_{r'',1}^{\Delta(r'')-\Delta(r')+1}$ for $r' \in\und{\lambda}_v$
\item $(\vv'''+\br{\vv''',\ww}\vv_{r,2}^1, j_b-1)$ where 
$$\vv''' = \sum_{r'\in \und{\lambda}_v, r'>r}\vv_{r',1}^{\Delta(r')+\nabla(r)-j_b+1}-\sum_{r'\geq j_b, s\in[1,2m(r')]} a_{r',s}\vv^{r'-j_b+1}_{r',s}$$
\item $(N\vv^1_{r,2}=\vv^2_{r,2}, r-1)$
\end{enumerate}
In any case, we have $\hat{\lambda}' = \hat{\lambda}\pch{r,j_b}{r-1,j_b-1}$. Thus direct calculation shows that $\mu'=\mu^{-[m(\geq j_b)+1, M(\mu)]}$ and $\nu'=(\nu^{+[m(\geq j_b)+1,M(\mu)]})\pch{\Delta(r)+1}{\Delta(r)} = \nu^{+[m(>j)+1,M(\mu)-1]}$
where $M(\mu)=m_\mu(\mathord{\geq \nabla(r)})$.
As $\#(\ker_{> j_{b+1}}N-\ker_{>j_b}N)^F=\qe{2m(\geq j_b)}-\qe{2m(> j_b)}$, it follows that
\begin{align*}
&\sum_{\ww-\vv_{r,1}^r \in (\ker_{> j_{b+1}}N-\ker_{>j_b}N)^F}\tspq{\mu',\nu'} 
\\&=(\qe{2m(\geq j_b)}-\qe{2m(> j_b)})\tspq{\mu^{-[m(\geq j_b)+1, M(\mu)]},\nu^{+[m(\geq j_b)+1,M(\mu)-1]}}.
\end{align*}

\section{Main theorem}
We summarize the results in the previous section and conclude our second main theorem. First, we recall some notations; see \ref{sec:weyl} for $W$ and $W'$; see \ref{sec:partition} for $\mu\pch{a_1,a_2,\ldots}{b_1, b_2, \ldots}$, $m(r)$, $m_{\geq r}$, etc.; see \ref{sec:misc} for $\qe{-}$; see \ref{sec:parexo} for $\fP_2$; see \ref{sec:standardexo} for  $\und{\lambda}_v$; see \ref{sec:greenexo} for $\tspq{\mu, \nu}$; see Section \ref{sec:calcexo} for $\Delta$ and $\nabla$.
\begin{thm}[Main theorem for exotic Springer representations] \label{thm:mainexo} For $(\mu, \nu) \in \fP_2$, let $\lambda=\mu+\nu$. Then the character $\Res^W_{W'} \tspq{\mu, \nu}$ is equal to
\begin{align*}
&\sum_{r \not \in\und{\lambda}_v}\frac{\qe{2m_{\geq r}}-\qe{2m_{>r}}}{q-1} \tspq{\mu,\nu\pch{\Delta(r)}{\Delta(r)-1}}\allowdisplaybreaks
\\+&\sum_{\substack{r \in \und{\lambda}_v\\\textup{case 2}}} \left[
\begin{aligned}
&\qe{2m_{\geq r}-1}\tspq{\mu^{+[m(\geq r)+1,M(\nu) ]},\nu^{-[m(\geq r),M(\nu) ]}}
\\&+\frac{\qe{2m_{\geq r}-1}-\qe{2m_{>r}+1}}{q-1}\tspq{\mu\pch{\nabla(r)}{\nabla(r)-1},\nu}
\\&+\qe{2m_{> r}}\tspq{\mu^{-[m(>r)+1, M(\mu)]},\nu^{+[m(>r)+1, M(\mu)-1]}}
\end{aligned}
\right]\allowdisplaybreaks
\\+&\sum_{\substack{r \in \und{\lambda}_v\\\textup{case 3}}} \left[
\begin{aligned}
&\left(\qe{2m_{\geq r}-1}-\qe{2m_{>r}-1}\right)\tspq{\mu^{+[m(\geq r)+1,M(\nu)]},\nu^{-[m(\geq r),M(\nu)]}}
\\&+\frac{\qe{2m_{\geq r}-1}-\qe{2m_{>r}-1}}{q-1}\tspq{\mu\pch{\nabla(r)}{\nabla(r)-1},\nu}
\end{aligned}
\right]\allowdisplaybreaks
\\+&\sum_{\substack{r \in \und{\lambda}_v\\\textup{case 4}}} \left[
\begin{aligned}
&\qe{2m_{\geq r}-1}\tspq{\mu^{+[m(\geq r)+1,M(\nu) ]},\nu^{-[m(\geq r),M(\nu) ]}}
\\&+ \frac{\qe{2m_{\geq r}-1}-\qe{2m_{>r}+1}}{q-1}\tspq{\mu\pch{\nabla(r)}{\nabla(r)-1},\nu}
\\&+\qe{2m_{>j}}\tspq{\mu^{-[m(>j)+1, M(\mu)]},\nu^{+[m(>j)+1,M(\mu)-1]}}
\\&+\sum_{k\in \und{\lambda}, r<k\leq j} \left(\qe{2m_{\geq k}}-\qe{2m_{> k}}\right)\tspq{\mu^{-[m(\geq k)+1, M(\mu)]},\nu^{+[m(\geq k)+1,M(\mu)-1]}}
\end{aligned}
\right].
\end{align*}
Here $m_{\geq r}$, $m(\geq r)$, etc. are defined with respect to $\lambda$. Also we set $M(\nu) = m_\nu(\geq \Delta(r))$ and $M(\mu) = m_\mu(\geq \nabla(r))$. We say
\begin{enumerate}[label=$-$]
\item \textnormal{case 2} if $\nabla(r')>\nabla(r), \Delta(r')>\Delta(r)$ for any $r' \in \und{\lambda}$ such that $r'>r$,
\item \textnormal{case 3} if there exists $r' \in \und{\lambda}$ such that $r'>r$ and $\Delta(r')=\Delta(r)$, and
\item \textnormal{case 4} if there exists $r' \in \und{\lambda}$ such that $r'>r$ and $\nabla(r')=\nabla(r)$.
\end{enumerate}
In \textnormal{case 4}, we set $j \in \und{\lambda}$ to be the largest value such that $\nabla(j)=\nabla(r)$ and $\Delta(r')\neq \Delta(j)$ for any $r' \in \und{\lambda}$ such that $r'>j$. (Such $j$ always exists and may equal $r$.)
We ignore terms whose coefficients are zero or when $\tspq{-,-}$ is not well-defined.
\end{thm}

If we evaluate the above theorem at $q=1$, then we have an ungraded version which is an analogue of \cite[Theorem 5.3]{kim:euler} for exotic Springer representations.

\begin{cor} Let $\tsp{1}{-,-}\colonequals \tspq(-,-)|_{q=1}$. For $(\mu, \nu) \in \fP_2$, let $\lambda=\mu+\nu$. Then the character $\Res^W_{W'} \tsp{1}{\mu, \nu}$ is equal to
\begin{align*}
&\sum_{r \not \in\und{\lambda}_v}2m_r \tsp{1}{\mu,\nu\pch{\Delta(r)}{\Delta(r)-1}}\allowdisplaybreaks
\\+&\sum_{\substack{r \in \und{\lambda}_v\\\textup{case 2}}} \left[
\begin{aligned}
&\tsp{1}{\mu^{+[m(\geq r)+1,M(\nu) ]},\nu^{-[m(\geq r),M(\nu) ]}}
\\&+(2m_r-2)\tsp{1}{\mu\pch{\nabla(r)}{\nabla(r)-1},\nu} +\tsp{1}{\mu^{-[m(>r)+1, M(\mu)]},\nu^{+[m(>r)+1, M(\mu)-1]}}
\end{aligned}
\right]\allowdisplaybreaks
\\+&\sum_{\substack{r \in \und{\lambda}_v\\\textup{case 3}}} 
2m_r\tsp{1}{\mu\pch{\nabla(r)}{\nabla(r)-1},\nu}
\allowdisplaybreaks
\\+&\sum_{\substack{r \in \und{\lambda}_v\\\textup{case 4}}} \left[
\begin{aligned}
&\tsp{1}{\mu^{+[m(\geq r)+1,M(\nu) ]},\nu^{-[m(\geq r),M(\nu) ]}}
\\&+(2m_r-2)\tsp{1}{\mu\pch{\nabla(r)}{\nabla(r)-1},\nu}+\tsp{1}{\mu^{-[m(>j)+1, M(\mu)]},\nu^{+[m(>j)+1,M(\mu)-1]}}
\end{aligned}
\right].
\end{align*}
Here $m_{ r}$, $m( r)$, etc. are defined with respect to $\lambda$. Also we set $M(\nu) = m_\nu(\geq \Delta(r))$ and $M(\mu) = m_\mu(\geq \nabla(r))$. We say
\begin{enumerate}[label=$-$]
\item \textnormal{case 2} if $\nabla(r')>\nabla(r), \Delta(r')>\Delta(r)$ for any $r' \in \und{\lambda}$ such that $r'>r$,
\item \textnormal{case 3} if there exists $r' \in \und{\lambda}$ such that $r'>r$ and $\Delta(r')=\Delta(r)$, and
\item \textnormal{case 4} if there exists $r' \in \und{\lambda}$ such that $r'>r$ and $\nabla(r')=\nabla(r)$.
\end{enumerate}
In \textnormal{case 4}, we set $j \in \und{\lambda}$ to be the largest value such that $\nabla(j)=\nabla(r)$ and $\Delta(r')\neq \Delta(j)$ for any $r' \in \und{\lambda}$ such that $r'>j$. (Such $j$ always exists and may equal $r$.)
We ignore terms whose coefficients are zero or when $\tspq{-,-}$ is not well-defined.
\end{cor}

\part{Further remarks}
\section{Equivalence of two main theorems} \label{sec:equiv}
In \cite{kat11}, Kato compared two Springer theories discussed in the previous parts, namely his exotic version and the one corresponding to $\Lie Sp_{2n} (\Ftbar)$, and showed that there is an equivalence of these two. His argument is based on some deformation of the exotic nilpotent cone. Here, in the same spirit, we show that our two main theorems are equivalent. Namely, we prove that there exists a bijection $\iota: \Omega \rightarrow \fP_2$ such that the total Springer representations of $(\lambda, \chi) \in \Omega$ and $\iota(\lambda, \chi)$ coincide, and under this bijection Theorem \ref{thm:main} and \ref{thm:mainexo} are equivalent. Note that the equality of these total Springer representations is already known to experts; for example see the closing remark of \cite{ahs11}.

\subsection{Bijection between $\Omega$ and $\fP_2$}
First we define a bijection $\iota : \Omega\rightarrow \fP_2$. (This bijection is deduced from, but not exactly the same as, the one defined in \cite[8.1]{xue12:comb}.)  For $(\lambda, \chi) \in \Omega$, choose $s \in \bN$ such that $2s\geq l(\lambda)$. 
We partition $[1, 2s+1]$ into blocks of size 1 or 2 such that:
\begin{enumerate}[label=$-$]
\item $\{i\}$ is a single block if and only if $\chi(\lambda_i)=\lambda_i/2 $ and
\item other blocks consist of two consecutive integers.
\end{enumerate}
Note that if $\{i, i+1\}$ is a block then $\lambda_i=\lambda_{i+1}$ (and thus $\chi(\lambda_i)=\chi(\lambda_{i+1})$). Now we set $c_i$ for $i \in [1, 2s+1]$ to be
\begin{enumerate}[label=$-$]
\item if $\{i\}$ is a single block then $c_i = \lambda_i/2$, and
\item otherwise if $\{i, i+1\}$ is a block then $c_i = \chi(\lambda_i)$ and $c_{i+1} =\lambda_i - \chi(\lambda_i)$.
\end{enumerate}
Now we set $\mu=(c_1, c_3, \ldots, c_{2s+1})$ and $\nu=(c_2, c_4, \ldots, c_{2s})$ (and remove zeroes at the end if necessary so that $\mu$ and $\nu$ do not depend on the choice of $s$). Then, from the definition of $\Omega$ it follows that $\mu$ and $\nu$ are partitions and $|\mu|+|\nu|=n$. Now we define $\iota(\lambda, \chi) = (\mu, \nu)$. 

We claim that $\iota$ is a bijection from $\Omega$ to $\fP_2$ with its inverse given as follows. (This is the same as the one defined in \cite[2.6]{xue12}.) For $(\mu, \nu) \in \fP_2$, we set
$$\lambda_1 = \left\{
\begin{aligned}
&\mu_1+\nu_1 && \textup{ if } \mu_1<\nu_1
\\&2\mu_1 && \textup{ if } \mu_1\geq \nu_1
\end{aligned}
\right. \qquad
\chi(\lambda_1)= \mu_1$$
and for $i \geq 1$ we set
\begin{align*}
\lambda_{2i} = &\left\{
\begin{aligned}
&\mu_{i+1}+\nu_i && \textup{ if } \nu_i<\mu_{i+1}
\\&\mu_{i}+\nu_i && \textup{ if }  \nu_i>\mu_{i}
\\&2\nu_i && \textup{ if } \mu_{i+1} \leq \nu_i \leq \mu_i
\end{aligned}
\right.
&\chi(\lambda_{2i})= &\left\{
\begin{aligned}
&\mu_i&& \textup{ if } \nu_i>\mu_{i}
\\&\nu_i && \textup{ if }  \nu_i\leq \mu_{i}
\end{aligned}\right.
\\\lambda_{2i+1} = &\left\{
\begin{aligned}
&\mu_{i+1}+\nu_i && \textup{ if } \mu_{i+1}>\nu_i
\\&\mu_{i+1}+\nu_{i+1} && \textup{ if }  \mu_{i+1}<\nu_{i+1}
\\&2\mu_{i+1} && \textup{ if } \nu_{i+1} \leq \mu_{i+1} \leq \nu_i
\end{aligned}
\right.
&\chi(\lambda_{2i+1})= &\left\{
\begin{aligned}
&\mu_{i+1}&& \textup{ if } \mu_{i+1}\leq\nu_i
\\&\nu_{i+1} && \textup{ if }  \mu_{i+1}>\nu_i
\end{aligned}\right.
\end{align*}
Then $\iota^{-1}(\mu, \nu) = (\lambda, \chi)$ where $\lambda$ and $\chi$ are described as above.

\begin{lem} The two maps defined above, namely $\iota$ and $\iota^{-1}$, are inverses to each other. In particular, $\iota$ is a bijection from $\Omega$ to $\fP_2$.
\end{lem}
\begin{proof} It essentially follows from \cite[8.1, Lemma]{xue12:comb}. We omit the details.
\end{proof}

\subsection{Shoji's limit symbol}
It is theoretically important to relate $\fP_2$, $\Omega$ with a well-known combinatorial object called symbols. After firstly defined by Lusztig, it has been playing a crucial role in representation theory of Weyl groups, finite groups of Lie type, etc. Here, instead of explaining the general theory of symbols, we simply focus on the version directly linked to our setting. These are called limit symbols which are extensively studied by Shoji \cite{sho04}.

Let us fix integers $r, s$ such that $r\geq s+n\geq 2n$. For an integer $m \in \bN$, we set $\zeta_m=(mr, (m-1)r, \ldots, 2r, r)$ and $\eta_m=(s+(m-1)r, s+(m-2)r, \ldots, s+2r, s+r,s)$. We set $Z^{r,s}(m) \colonequals\{(\zeta_m+\mu, \eta_m+\nu)\mid (\mu, \nu)\in \fP_2, l(\mu)\leq m+1, l(\nu)\leq m\}$. Let us define an equivalence relation $\approx$ on $\bigsqcup_{m \in \bN} Z^{r,s}(m)$, where $(\alpha, \beta) \approx (\alpha', \beta')$ for $(\alpha, \beta) \in Z^{r,s}(m)$ and $(\alpha', \beta') \in Z^{r,s}(m')$ if there exists $(\mu, \nu) \in \fP_2$ such that $(\zeta_m+\mu, \eta_m+\nu)=(\alpha,\beta)$ and $(\zeta_{m'}+\mu, \eta_{m'}+\nu)=(\alpha',\beta')$. We set $\bar{Z}^{r,s}$ to be the set of equivalence classes in $\bigsqcup_{m \in \bN} Z^{r,s}(m)$. It is clear that there exists a canonical bijection $\tau: \fP_2 \simeq \bar{Z}^{r,s}$, which also induces another bijection $\tau\circ\iota: \Omega \simeq \bar{Z}^{r,s}$. The elements of $\bar{Z}^{r,s}$ are called limit symbols.

\begin{rmk} The bijection defined in \cite[8.1]{xue12:comb} coincides with $\tau\circ\iota$ when $r=2n+2$ and $s=n+1$. Note that the parameters $r$ and $s$ therein are equal to $r-s$ and $s$ in our setting, respectively.
\end{rmk}

\subsection{Equality of total Springer representations}

There is an algorithm established by Shoji \cite{sho83} and Lusztig \cite{lus86}, now commonly called the Lusztig-Shoji algorithm,  originally invented to calculate the (generalized) Green functions of reductive groups in good characteristic. Later it is generalized \cite{sho01, sho02, sho04} so that it attaches Green functions to any symbol associated to complex reflection groups of the form $G(r, p, n)$. Here we only focus on limit symbols associated to $W$ and their Green functions. 

For any limit symbol $X \in \bar{Z}^{r,s}$ and an indeterminate $x$, we set $\tspx{X}$ to be the $\bC(x)$-valued character of $W$ such that $\tspx{X}(w)$ is the Green function defined in \cite[Section 5]{sho01} (see also \cite[Section 3]{sho04}). We abuse terminology and also call it the total Springer representation corresponding to the symbol $X$. Then it does not depend on the choice of $r$ and $s$ provided that $r\geq s+n\geq 2n$, and it follows from \cite[Proposition 3.3]{sho04} that in fact we have $\im \tspx{X} \subset \bQ[x]$, i.e. $\tspx{X}$ is a $\bQ[x]$-valued character of $W$. Also, we have the following equality of total Springer representations. (This theorem is essentially equivalent to \cite[Conjecture 6.4]{ah08}.)
\begin{thm} For any $(\lambda, \chi) \in \Omega$, we have $\tspq{\lambda, \chi}=\tspx{\tau\circ\iota(\lambda, \chi)}|_{x=q}$. In particular, $\tspq{\lambda, \chi}$ is a ``polynomial in $q$''.
\end{thm}
\begin{proof} This follows from \cite{ss14} or \cite{kat17}.
\end{proof}

There is also an analogous statement for the total Springer representations of $\Lie Sp_{2n}(\Ftbar)$.
\begin{thm} For any $(\mu, \nu) \in \fP_2$, we have $\tspq{\mu, \nu}=\tspx{\iota(\mu, \nu)}|_{x=q}$. In particular, $\tspq{\mu,\nu}$ is a ``polynomial in $q$''.
\end{thm}
\begin{proof} By the result of \cite{spa82} (see also \cite[Section 8]{xue12:comb}), the Springer correspondence of $\Lie Sp_{2n}(\Ftbar)$ is governed by combinatorics of $\bar{Z}^{2n+2, n+1}$. (Note that the convention of \cite{xue12:comb} is slightly different from ours in a sense that $(r,s)=(n+1, n+1)$ therein is equivalent to $(r,s)=(2n+2, n+1)$ in our setting.) Then one can follow the argument of \cite[Section 24]{lus86} or equivalently \cite[Section 6]{ach11}. It is assumed therein that the characteristic of the base field is good (i.e. different from 2 in our case), but their proof is still valid except that the resulting Green functions are not guaranteed to be contained in $\bQ[x]$ but in $\bQ(x^{1/2})$. (See also \cite[Remark 6.2]{ach11}.)  However, we already know the polynomiality of $\tspx{\iota(\mu, \nu)}$, thus the second part of the statement also follows.
\end{proof}

From the two theorems above it follows that ``$\tspq{\lambda, \chi} = \tspq{\iota(\lambda, \chi)}$'', but one needs to be careful since $q$ of $\tspq{\lambda, \chi}$ is a power of $2$ whereas that of $\tspq{\iota(\lambda, \chi)}$ is assumed to be a power of an odd prime. Instead, we define $\tspx{\lambda,\chi}$ (resp. $\tspx{\mu, \nu}$) to be $\tspx{\tau\circ\iota(\lambda, \chi)}$ (resp. $\tspx{\iota(\mu, \nu)}$). They are $\bQ[x]$-valued character of $W$. Then we can state the following corollary, which is merely a summary of the two theorems above.

\begin{cor} \label{cor:greenequal} For $(\lambda, \chi) \in \Omega$, we have $\tspx{\lambda, \chi} = \tspx{\iota(\lambda, \chi)}$.
\end{cor}

\subsection{Equivalence of two main theorems} 
By the result above, we have:
\begin{thm} Suppose that $(\lambda, \chi) \in \Omega$ and $(\mu, \nu) = \iota(\lambda, \chi)\in \fP_2$. Then the formula of Theorem \ref{thm:main} is equivalent to that of Theorem \ref{thm:mainexo}. More precisely, if we replace each $\tspq{-, -}$ with $\tspx{\iota(-,-)}$, $\qe{-}$ with $\xe{-}$, etc. in the formula of Theorem \ref{thm:main} then it coincides with the formula of Theorem \ref{thm:mainexo} after replacing each $\tspq{-, -}$ with $\tspx{-,-}$, $\qe{-}$ with $\xe{-}$, etc.
\end{thm}
\begin{proof} By Corollary \ref{cor:greenequal}, the two formulas of Theorem \ref{thm:main}
and \ref{thm:mainexo} should give the same $\bQ[x]$-valued character of $W'$ (after replacing $q$ with $x$). Now the result follows from the fact that $\{\tspx{\lambda, \chi} \mid (\lambda, \chi) \in \Omega\}$ and $\{\tspx{\mu, \nu}\mid (\mu, \nu) \in \fP_2\}$ are linearly independent sets in the vector space of $\bQ(x)$-valued characters of $W$ for any $n \in \bZ_{>0}$.
\end{proof}

\begin{rmk} It is an interesting exercise to prove the above statement based only on the combinatorial description of $\iota$, which we leave to the reader.
\end{rmk}

\section{Examples} \label{sec:example}
Here we provide some examples of our main theorems \ref{thm:main} and \ref{thm:mainexo}. In this section, we write $\cdots 2^{m_2}1^{m_1}$ instead of $\lambda$ (where $m_r=m_\lambda(r)$) and $\cdots 2^{m_2}_{\chi(2)}1^{m_1}_{\chi(1)}$ instead of $(\lambda, \chi)$ for simplicity. When $n=2$, we have
\begin{align*}
\Res^{W}_{W'}\tspq{4^1_2}&=\tspq{2^1_1}
\\\Res^{W}_{W'}\tspq{2^1_11^2_0}&=(q^2+q)\tspq{2^1_1}+\tspq{1^2_0}
\\\Res^{W}_{W'}\tspq{2^2_1}&=q\tspq{2^1_1}+\tspq{1^2_0}
\\\Res^{W}_{W'}\tspq{2^2_0}&=(q+1)\tspq{1^2_0}
\\\Res^{W}_{W'}\tspq{1^4_0}&=(q^3+q^2+q+1)\tspq{1^2_0}
\end{align*}
or equivalently
\begin{align*}
\Res^{W}_{W'}\tspq{2^1,\emptyset}&=\tspq{1^1,\emptyset}
\\\Res^{W}_{W'}\tspq{1^2,\emptyset}&=(q^2+q)\tspq{1^1, \emptyset}+\tspq{\emptyset, 1^1}
\\\Res^{W}_{W'}\tspq{1^1,1^1}&=q\tspq{1^1, \emptyset}+\tspq{\emptyset, 1^1}
\\\Res^{W}_{W'}\tspq{\emptyset, 2^1}&=(q+1)\tspq{\emptyset, 1^1}
\\\Res^{W}_{W'}\tspq{\emptyset, 1^2}&=(q^3+q^2+q+1)\tspq{\emptyset, 1^1}
\end{align*}
If $n=3$, then we have
\begin{align*}
\Res^{W}_{W'}\tspq{6^1_3}&=\tspq{4^1_2}
\\\Res^{W}_{W'}\tspq{4^1_21^2_0}&=(q^2 + q)\tspq{4^1_2}+\tspq{2^1_11^2_0}
\\\Res^{W}_{W'}\tspq{2^1_11^4_0}&=(q^4 + q^3 + q^2 + q)\tspq{2^1_11^2_0}+\tspq{1^4_0}
\\\Res^{W}_{W'}\tspq{4^1_22^1_1}&=q\tspq{4^1_2}+\tspq{2^2_1}
\\\Res^{W}_{W'}\tspq{2^3_1}&=(q^2-1)\tspq{2^2_1}+(q+1)\tspq{2^1_11^2_0}+\tspq{2^2_0}
\\\Res^{W}_{W'}\tspq{3^3_1}&=q\tspq{2^2_1}+\tspq{2^2_0}
\\\Res^{W}_{W'}\tspq{2^2_11^2_0}&=(q^3 + q^2)\tspq{2^2_1}+q\tspq{2^1_11^2_0}+\tspq{1^4_0}
\\\Res^{W}_{W'}\tspq{3^2_0}&=(q + 1)\tspq{2^2_0}
\\\Res^{W}_{W'}\tspq{2^3_1}&=(q^3 + q^2)\tspq{2^2_0}+(q+1)\tspq{1^4_0}
\\\Res^{W}_{W'}\tspq{2^3_1}&=(q^5 + q^4 + q^3 + q^2 + q + 1)\tspq{1^4_0}
\end{align*}
or equivalently
\begin{align*}
\Res^{W}_{W'}\tspq{3^1,\emptyset}&=\tspq{2^1, \emptyset}
\\\Res^{W}_{W'}\tspq{2^11^1,\emptyset}&=(q^2 + q)\tspq{2^1,\emptyset}+\tspq{1^2,\emptyset}
\\\Res^{W}_{W'}\tspq{1^3, \emptyset}&=(q^4 + q^3 + q^2 + q)\tspq{1^2,\emptyset}+\tspq{\emptyset, 1^2}
\\\Res^{W}_{W'}\tspq{2^1,1^1}&=q\tspq{2^1, \emptyset}+\tspq{1^1,1^1}
\\\Res^{W}_{W'}\tspq{1^2, 1^1}&=(q^2-1)\tspq{1^1,1^1}+(q+1)\tspq{1^2, \emptyset}+\tspq{\emptyset, 2^1}
\\\Res^{W}_{W'}\tspq{1^1,2^1}&=q\tspq{1^1,1^1}+\tspq{\emptyset, 2^1}
\\\Res^{W}_{W'}\tspq{1^1,1^2}&=(q^3 + q^2)\tspq{1^1,1^1}+q\tspq{1^2,\emptyset}+\tspq{\emptyset, 1^2}
\\\Res^{W}_{W'}\tspq{\emptyset,3^1}&=(q + 1)\tspq{\emptyset, 2^1}
\\\Res^{W}_{W'}\tspq{\emptyset, 2^11^1}&=(q^3 + q^2)\tspq{\emptyset, 2^1}+(q+1)\tspq{\emptyset, 1^2}
\\\Res^{W}_{W'}\tspq{\emptyset, 1^3}&=(q^5 + q^4 + q^3 + q^2 + q + 1)\tspq{\emptyset, 1^2}
\end{align*}

From now on we write $\chiq{-,-}\colonequals \tspq{-,-}(id)$ and $\chibq{-,-}\colonequals \tspq{-,-}(s_1)$ (see Section \ref{sec:weyl} for the definition of $s_1 \in W$). When $n=1$, we have the base case
$$\chiq{2^1_1}=\chibq{2^1_1}=1,\quad \chiq{1^2_0}=q+1,\quad \chibq{1^2_0}=-q+1$$
or equivalently
$$\chiq{1^1, \emptyset}=\chibq{1^1, \emptyset}=1,\quad \chiq{\emptyset, 1^1}=q+1,\quad \chibq{\emptyset, 1^1}=-q+1$$
Therefore, from the equations above it follows that when $n=2$
\begin{align*}
\chiq{4^1_2}=\chiq{2^1,\emptyset}&=1
\\\chiq{2^1_11^2_0}=\chiq{1^2,\emptyset}&=q^2 + 2q + 1
\\\chiq{2^2_1}=\chiq{1^1,1^1}&=2q + 1
\\\chiq{2^2_0}=\chiq{\emptyset, 2^1}&=q^2 + 2q + 1
\\\chiq{1^4_0}=\chiq{\emptyset, 1^2}&=q^4 + 2q^3 + 2q^2 + 2q + 1\allowdisplaybreaks
\\\chibq{4^1_2}=\chibq{2^1,\emptyset}&=1
\\\chibq{2^1_11^2_0}=\chibq{1^2,\emptyset}&=q^2 + 1
\\\chibq{2^2_1}=\chibq{1^1,1^1}&=1
\\\chibq{2^2_0}=\chibq{\emptyset, 2^1}&= -q^2 + 1
\\\chibq{1^4_0}=\chibq{\emptyset, 1^2}&=-q^4 + 1
\end{align*}
and when $n=3$
\begin{align*}
\chiq{6^1_3}=\chiq{3^1,\emptyset}&= 1
\\\chiq{4^1_21^2_0}=\chiq{2^11^1,\emptyset}&=2q^2 + 3q + 1
\\\chiq{2^1_11^4_0}=\chiq{1^3, \emptyset}&=q^6 + 3q^5 + 5q^4 + 6q^3 + 5q^2 + 3q + 1
\\\chiq{4^1_22^1_1}=\chiq{2^1,1^1}&=3q + 1
\\\chiq{2^3_1}=\chiq{1^2, 1^1}&=3q^3 + 5q^2 + 3q + 1
\\\chiq{3^3_1}=\chiq{1^1,2^1}&=3q^2 + 3q + 1
\\\chiq{2^2_11^2_0}=\chiq{1^1,1^2}&=3q^4 + 6q^3 + 5q^2 + 3q + 1
\\\chiq{3^2_0}=\chiq{\emptyset,3^1}&=q^3 + 3q^2 + 3q + 1
\\\chiq{2^3_1}=\chiq{\emptyset, 2^11^1}&=2q^5 + 6q^4 + 7q^3 + 5q^2 + 3q + 1
\\\chiq{2^3_1}=\chiq{\emptyset, 1^3}&=q^9 + 3q^8 + 5q^7 + 7q^6 + 8q^5 + 8q^4 + 7q^3 + 5q^2 + 3q + 1\allowdisplaybreaks
\\\chibq{6^1_3}=\chibq{3^1,\emptyset}&= 1
\\\chibq{4^1_21^2_0}=\chibq{2^11^1,\emptyset}&=2q^2 + q + 1
\\\chibq{2^1_11^4_0}=\chibq{1^3, \emptyset}&= q^6 + q^5 + q^4 + 2q^3 + q^2 + q + 1
\\\chibq{4^1_22^1_1}=\chibq{2^1,1^1}&=q + 1
\\\chibq{2^3_1}=\chibq{1^2, 1^1}&=q^3 + q^2 + q + 1
\\\chibq{3^3_1}=\chibq{1^1,2^1}&=-q^2 + q + 1
\\\chibq{2^2_11^2_0}=\chibq{1^1,1^2}&= -q^4 + 2q^3 + q^2 + q + 1
\\\chibq{3^2_0}=\chibq{\emptyset,3^1}&=-q^3 - q^2 + q + 1
\\\chibq{2^3_1}=\chibq{\emptyset, 2^11^1}&= -2q^5 - 2q^4 + q^3 + q^2 + q + 1
\\\chibq{2^3_1}=\chibq{\emptyset, 1^3}&=-q^9 - q^8 - q^7 - q^6 + q^3 + q^2 + q + 1
\end{align*}
We can iterate this process to calculate $\chiq{-,-}, \chibq{-,-}$ for large $n$. For example, we have
\begin{align*}
\chiq{5^13^11^1, 4^12^1}&=7297290q^{20} + 22783761q^{19} + 36074610q^{18} + 39588549q^{17} + 34435440q^{16} 
\\&\quad+ 25561211q^{15} + 16909620q^{14} + 10213320q^{13} + 5704674q^{12} + 2963833q^{11} 
\\&\quad+ 1434006q^{10} + 644722q^9 + 268005q^8 + 102202q^7 + 35356q^6 + 10918q^5 
\\&\quad+ 2939q^4 + 665q^3 + 119q^2 + 15q + 1, \allowdisplaybreaks
\\\chibq{5^13^11^1, 4^12^1}&=1459458q^{20} + 4711707q^{19} + 7877298q^{18} + 9184175q^{17} + 8489624q^{16}
\\&\quad+ 6697093q^{15} + 4730908q^{14} + 3068364q^{13} + 1847598q^{12} + 1036607q^{11} 
\\&\quad+ 541814q^{10} + 263094q^9 + 118145q^8 + 48750q^7 + 18304q^6 + 6162q^5 
\\&\quad+ 1819q^4 + 455q^3 + 91q^2 + 13q + 1.
\end{align*}

\section{On an affine paving of Springer fibers}\label{sec:affinepav}
For a variety $X$, we say that $X$ admits an affine paving if there exist pairwise disjoint subvarieties $Y_1, Y_2, \ldots, Y_k$ of $X$ such that $\bigsqcup_{i=1}^j Y_i$ is closed for any $j \in [1, k]$ and each $Y_i$ is isomorphic to an affine space. Now suppose that $X$ and $Y_1, \ldots, Y_k$ are defined over some finite field $\Fq$ and $F$ is the corresponding geometric Frobenius acting on each of them. Then it is easy to show that $H^i(X)=0$ when $i$ is odd and the eigenvalues of $F^*$ on $H^{2i}(X)$ are equal to $q^i$.

It is a well-known fact that any Springer fiber of classical groups in good characteristic admits an affine paving. It is proved by \cite{dclp} when the base field is $\bC$, but their argument can be generalized to good characteristic setting; see \cite[Section 11]{jan04}. However, to the best of the author's knowledge it is still not known in the case of bad characteristic.

On the other hand, it is known that
\begin{thm}[{\cite[Theorem 1.1]{mau17}}] An exotic Springer fiber admits an affine paving over any algebraically closed field.
\end{thm}
Recall that the total Springer representations for the exotic case are defined by $\tspq{\mu, \nu}:W \rightarrow \qlbar: w \mapsto \sum_{i\in \bZ}(-1)^i \tr(wF^*, H^i(\cB_{N,v}))$ for each $(\mu, \nu) \in \fP_2$. Assume as in Part \ref{part:exotic} that $\kk$ is an algebraically closed field of $\bF_p$ for some odd prime $p$. Suppose that $(N, v) \in \cN^- \times V$ is defined over $\Fq$ where $q$ is a power of $p$ and that its orbit is parametrized by $(\mu,\nu)\in \fP_2$. Then by replacing $q$ by its (sufficiently large) power if necessary, we may assume that an affine paving of $\cB_{N,v}$ is defined over $\Fq$. Then it follows that $H^i(\cB_{N,v})=0$ if $i$ is odd and the eigenvalues of $F^*$ on $H^{2i}(\cB_{N,v})$ equals $q^{i}$. In other words, we have $\tspq{\mu, \nu}(w)=\sum_{i\in \bZ} \tr(w, H^{2i}(\cB_{N,v}))q^i$. From this we may also prove that $\tspq{\mu, \nu}$ is ``a polynomial in $q$'', without using \cite[Proposition 3.3]{sho04}.

On the other hand, it is still not known that a Springer fiber associated to $\Lie Sp_{2n}(\Ftbar)$ admits an affine paving except some trivial cases. Here we discuss some partial results which can be obtained from (the proof of) Theorem \ref{thm:main}. We start with the following lemma.

\begin{lem} \label{lem:affine}Suppose that $(\lambda, \chi) \in \Omega$ satisfies the following condition.
\begin{enumerate}
\item If $r\in \und{\lambda}$ is not critical, then either $\chi(r)=0$ or there exist $r',r'' \in \und{\lambda}$ such that $r'<r<r''$, $\chi(r)=\chi(r')$, and $r''-\chi(r'')=r-\chi(r)$. 
\item For $r\in \und{\lambda}$, if $m_\lambda(r)$ is odd then $m_\lambda(r)=1$. 
\end{enumerate}
Let $\Fq$ be a finite field of characteristic 2 and $F$ be the corresponding geometric Frobenius. Assume that a nilpotent element $N\in (\Lie Sp_{2n}(\Ftbar))^F$ is contained in the $Sp_{2n}$-orbit parametrized by $(\lambda, \chi)$. Let $\pi: \cB_N \rightarrow \bP(\ker N)$ be the canonical projection. 
Then there exists an affine paving $Y_1, \ldots, Y_k$ of $\bP(\ker N)$ defined over $\Fq$ such that $\pi^{-1}(Y_i) \simeq \pi^{-1}(l_i) \times Y_i$ as $\Fq$-varieties for some/any $l_i \in Y_i^F$.
\end{lem}
\begin{proof} To this end we recall the calculations in Section \ref{sec:calculation} and check case-by-case. Here we only describe an affine paving of $\bP(\ker N)$ and leave to the reader that this paving satisfies the conditions in the statement.
\begin{enumerate}[leftmargin=*]
\item[(\ref{sec:case1})] Under our assumption, we only need to consider \textbf{Case 1} thereof. In this case, we may set the stratification of $\bP(\ker_{\geq r}N)-\bP(\ker_{>r}N)$ by $H_1, H_2, \ldots, H_{m(r)}$ where each $H_i$ is defined by $a_{r,1}=\cdots=a_{r,i-1}=0$ and $a_{r,i}\neq 0$.
\item[(\ref{sec:case2})] Under our assumption we have $\und{\lambda}\cap[r,j] = \{r\}$. We set the stratification of $\bP(\ker_{\geq r}N)-\bP(\ker_{>r}N)$ by $\bP(\ker_{\geq r}N)-H$, $H_3, H_4, \ldots, H_{m(r)}$, $H'-\bP(\ker_{>r}N)$, where $H, H'$ are as in \ref{sec:case2} and each $H_i$ is defined by $a_{r,1}=0$, $a_{r,3}=\cdots=a_{r,i-1}=0$, and $a_{r,i}\neq 0$.
\item[(\ref{sec:case3})] In this case we have $m(r)=1$ and also $\und{\lambda}\cap[r,j] = \{r\}$. Here $\bP(\ker_{\geq r}N)-\bP(\ker_{>r}N)$ itself defines such a stratification.
\item[(\ref{sec:case4})] Similarly to above we have $\und{\lambda}\cap[r,j] = \{r\}$. We set the stratification of $\bP(\ker_{\geq r}N)-\bP(\ker_{>r}N)$ by $\bP(\ker_{\geq r}N)-H$, $H_3, H_4, \ldots, H_{m(r)}$, $H'-\bP(\ker_{>r}N)$, where $H, H'$ are as in \ref{sec:case4} and each $H_i$ is defined by $a_{r,1}=0$, $a_{r,3}=\cdots=a_{r,i-1}=0$, and $a_{r,i}\neq 0$.
\end{enumerate}
Now the desired affine paving of $\bP(\ker N)$ is defined by taking the union of all the cases above.
\end{proof}
\begin{rmk}
Under the same assumption one can check that all the coefficients which appear in the formula of Theorem \ref{thm:main} are sums of powers of $q$, which is consistent with the statement of Lemma \ref{lem:affine}.
\end{rmk}
Now we use Lemma \ref{lem:affine} to prove the existence of an affine paving in some special cases.
\begin{thm} Suppose that a nilpotent element $N\in (\Lie Sp_{2n}(\Ftbar))^F$ is contained in the $Sp_{2n}$-orbit parametrized by $(\lambda, \chi)\in \Omega$. Assume that either $\lambda_3\leq 1$ or $\chi=0$. Then $\cB_N$ admits an affine paving.
\end{thm}
\begin{proof} First we suppose that $\lambda_3\leq 1$, i.e. $\lambda=(a, b, 1, \ldots, 1)$ for some $a\geq b\geq 1$ or $l(\lambda)\leq 2$. Then $(\lambda, \chi)$ satisfies the assumption of Lemma \ref{lem:affine} and any possible $\lambda\pch{r,r}{r-1,r-1}$ or $\lambda\pch{r}{r-2}$ for $r\in \und{\lambda}$ is of the same form as $\lambda$. Therefore the result follows from Lemma \ref{lem:affine} and induction on $n$. The other case is similar.
\end{proof}

From now on, let us attempt to apply the argument in \cite{dclp} to our setting. Namely, suppose that $N$ is defined as in \ref{sec:standard} and it is in the $Sp_{2n}$-orbit parametrized by $(\lambda, \chi) \in \Omega$. Let us define an action of $\bG_m$ (1-dimensional torus over $\Ftbar$) on $V$ as follows.
\begin{enumerate}
\item If $r\in \und{\lambda}$ is not critical, then set $\alpha \cdot \vv^t_{r,2i-1} =\alpha^{c_{r,i}-r-1+2t} \vv^t_{r,2i-1}$ and $\alpha\cdot \vv^t_{r,2i}= \alpha^{-c_{r,i}-r-1+2t}\vv^t_{r,2i}$ for some $c_{r,i} \in \bZ$ for $i \in [1, m_\lambda(r)/2]$.
\item If $m_\lambda(r)$ is odd, then set $\alpha \cdot \vv^t_{r,2i} =\alpha^{c_{r,i}-r-1+2t} \vv^t_{r,2i}$ and $\alpha\cdot \vv^t_{r,2i+1}= \alpha^{-c_{r,i}-r-1+2t}\vv^t_{r,2i+1}$ for some $c_{r,i} \in \bZ$ for $i \in [1, (m_\lambda(r)-1)/2]$, and $\alpha\cdot \vv^{t}_{r,1}= \alpha^{-r-1+2t}\vv^t_{r,1}$.
\item Otherwise, i.e. if $r$ is critical and $m_\lambda(r)$ is even, then set 
$\alpha \cdot \vv^t_{r,2i-1} =\alpha^{c_{r,i}-r-1+2t} \vv^t_{r,2i-1}$ and $\alpha\cdot \vv^t_{r,2i}= \alpha^{-a_{r,i}-r-1+2t}\vv^t_{r,2i}$ for some $c_{r,i} \in \bZ$ for $i \in [2, m_\lambda(r)/2]$, and $\alpha\cdot \vv^{t}_{r,1}=\alpha^{-2\chi(r)-1+2t}\vv^{t}_{r,1}$ and $\alpha\cdot \vv^{t}_{r,2}=\alpha^{2\chi(r)-2r-1+2t}\vv^{t}_{r,2}$,
\end{enumerate}
We extend this $\bG_m$-action to $V$ by linearity. Then direct calculation shows that $\bG_m$ preserves $\br{\ , \ }$ on $V$, thus it gives a cocharacter $\gamma: \bG_m \rightarrow Sp_{2n}$. This defines the weight decomposition $\fg = \bigoplus_{i \in \bZ}\fg_i$ and direct calculation shows that $N \in \fg_2$. Let us define $\fp \colonequals \bigoplus_{i \geq 0} \fg_i$ and $P \subset G$ be the parabolic subgroup satisfying $\Lie P = \fp$.

From now on, we assume that $c_{r,i}$ in the definition of $\gamma$ are chosen generically. Then we are in the setting analogous to \cite[1.12]{dclp}, i.e. the parabolic subgroup $P$ is said to be canonically attached to $N$. Also, let $L' \colonequals G'\times H'\subset G$, where
$$G'\colonequals Sp(\spn{\{a_{r,1} \mid r\in \und{\lambda}, m_\lambda(r) \textup{ odd} \} \cup\{a_{r,i} \mid  r\in \und{\lambda}, i\in \{1,2\}, r \textup{ critical}, m_\lambda(r) \textup{ even}\}})$$
and $H'$ is the product of the following subgroups $H'_r$ of $G$.
\begin{enumerate}
\item If $r\in \und{\lambda}$ is not critical, then set $H'_r =G\cap \prod_{i \in[1,m(r)]}GL(\spn{\vv^t_{r,i} \mid t\in [1,r]})$.
\item If $m_\lambda(r)$ is odd, then set $H'_r =G\cap \prod_{i \in[2,m(r)]}GL(\spn{\vv^t_{r,i} \mid t\in [1,r]})$.
\item Otherwise,  set $H'_r =G\cap \prod_{i \in[3,m(r)]}GL(\spn{\vv^t_{r,i} \mid t\in [1,r]})$.
\end{enumerate}
Then it is clear that $H'$ is a product of general linear groups, $L'$ is a Levi subgroup of $G$, and also $N \in \Lie L'$. Furthermore, one can check that $N$ is a distinguished nilpotent element in $\Lie L'$ (cf. \cite[Proposition 5.3.(ii)]{ls12}) and the centralizer of the $\gamma$-action in $G$ is contained in $L'$. 

In order to use argument of \cite{dclp}, the following step is essential:
\begin{conj}[{cf. \cite[Proposition 3.2]{dclp}}]\label{conj:smooth} The intersection of $\cB_N$ with each $P$-orbit in $\cB$ is smooth. (See also \cite[11.18]{jan04}.)
\end{conj}
A similar statement holds in good characteristic. It follows from the fact that the $P$-orbit of $N$ in $\bigoplus_{i \geq 2}\fg_2$ is dense (cf. \cite[5.9]{jan04}), which is no longer true in bad characteristic in general even when $N$ is distinguished. However, for Conjecture \ref{conj:smooth} it suffices to check the following weaker statement.
\begin{conj} The intersection of the $P$-orbit of $N$ in $\bigoplus_{i \geq 2}\fg_2$ with any linear subspace is smooth.
\end{conj}
If the $P$-orbit of $N$ in $\bigoplus_{i \geq 2}\fg_2$ is dense, then its intersection with any linear subspace is still open in such a linear subspace, which implies the conjecture above. (A similar condition appears in \cite[5.1(f)]{lus05}.)

On the other hand, if we assume the conjecture then we have the following;
\begin{thm} Suppose that Conjecture \ref{conj:smooth} holds for any nilpotent element of $\Lie Sp_{2n}(\Ftbar)$ for any $n \in \bN$. Then any Springer fiber of $\Lie Sp_{2n}(\Ftbar)$ admits an affine paving.
\end{thm}
\begin{proof} If Conjecture \ref{conj:smooth} is true then one can use similar argument to \cite{dclp} and show that it suffices only to consider distinguished nilpotent elements. Now if $(\lambda, \chi)\in \Omega$ parametrizes a distinguished nilpotent orbit, then from the description of $L'$ above we can deduce that any $r \in \und{\lambda}$ is critical and $m_\lambda(r)\leq 2$. (cf. \cite[Proposition 5.3.(ii)]{ls12}) Now we use Lemma \ref{lem:affine} and induction on $n$ to finish the proof.
\end{proof}

\section{Some questions} \label{sec:question}
We conclude with some questions which naturally arise from our discussion and results. Firstly, in this paper we only covered the Springer theory for $\Lie Sp_{2n}(\Ftbar)$ (along with exotic cases), but as discussed in the introduction there are more Springer theories for classical types in bad characteristic. We expect that similar geometric argument can be applied to other cases as well. It is likely that the answer of the following question is positive.
\begin{question} Is it possible to find a formula similar to the main theorems for any Springer theory in classical types? (The good characteristic case is discussed in \cite{kim:betti}.)
\end{question}

On the other hand, observe that the statements themselves of Theorem \ref{thm:main} and \ref{thm:mainexo} have a combinatorial flavor and do not involve with geometry of Springer fibers. As mentioned in Section \ref{sec:equiv}, the total Springer representations in this case can be obtained in a purely combinatorial manner using symbols and the Lusztig-Shoji algorithm \cite{sho04}. (A similar statement also holds for the total Springer representations (for classical types) in good characteristic, which was the original motivation of the Lusztig-Shoji algorithm \cite{sho83, lus86}.) At this moment it is not clear to the author whether our main results can be deduced directly from combinatorial method. In this point of view, the following question is natural to ask.
\begin{question} Is there any other, mostly combinatorial, proof of Theorem \ref{thm:main} and \ref{thm:mainexo}?
\end{question}

If the answer of the above question is yes, then hopefully one may also generalize it to other types of symbols. For instance, Shoji \cite{sho01, sho02} defined the Hall-Littlewood functions for symbols of complex reflection groups of the form $G(r,p,n)$, which allowed him to define the Green functions attached to symbols as a transition matrix between power symmetric functions and Hall-Littlewood functions (up to degree), similar to the original definition of Green polynomials in type $A$. If one can find a combinatorial proof of our main theorems (and also the result of \cite{kim:betti}), then it is likely that this method can be generalized to other cases. Namely, we may ask the following.
\begin{question} Is there any formula similar to Theorem \ref{thm:main}, \ref{thm:mainexo}, and \cite[Theorem 3.1]{kim:betti} for Green functions attached to symbols for complex reflection groups discussed in \cite{sho01, sho02}?
\end{question}

\bibliographystyle{amsalphacopy}
\bibliography{sprchar2}

\providecommand{\bysame}{\leavevmode\hbox to3em{\hrulefill}\thinspace}
\providecommand{\MR}{\relax\ifhmode\unskip\space\fi MR }
\providecommand{\MRhref}[2]{%
  \href{http://www.ams.org/mathscinet-getitem?mr=#1}{#2}
}
\providecommand{\href}[2]{#2}
\begin{thebibliography}{dCLP88}

\bibitem[Ach11]{ach11}
Achar, P.~N., \emph{Green functions via hyperbolic localization}, Doc. Math.
  \textbf{16} (2011), no.~869--884.

\bibitem[AH08]{ah08}
Achar, P.~N. and Henderson, A., \emph{Orbit closures in the enhanced nilpotent
  cone}, Adv. Math. \textbf{219} (2008), 27--62.

\bibitem[AHS11]{ahs11}
Achar, P.~N., Henderson, A., and Sommers, E., \emph{Pieces of nilpotent cones
  for classical groups}, Represent. Theory \textbf{15} (2011), 584--616.

\bibitem[BM83]{bm83}
Borho, W. and MacPherson, R., \emph{Partial resolutions of nilpotent
  varieties}, Ast{\'e}risque \textbf{101} (1983), 23--74.

\bibitem[dCLP88]{dclp}
de~Concini, C., Lusztig, G., and Procesi, C., \emph{Homology of the zero-set of
  a nilpotent vector field on a flag manifold}, J. Amer. Math. Soc. \textbf{1}
  (1988), no.~1, 15--34.

\bibitem[DM91]{dm91}
Digne, F. and Michel, J., \emph{Representations of {F}inite {G}roups of {L}ie
  {T}ype}, London {M}athematical {S}ociety {S}tudent {T}exts, vol.~21,
  Cambridge {U}niversity {P}ress, 1991.

\bibitem[Hes79]{hes79}
Hesselink, W.~H., \emph{Nilpotency in classical groups over a field of
  characteristic 2}, Math. Z. \textbf{166} (1979), 165--181.

\bibitem[Jan04]{jan04}
Jantzen, J.~C., \emph{Nilpotent {O}rbits in {R}epresentation {T}heory}, Lie
  {T}heory: {L}ie {A}lgebras and {R}epresentations (Anker, J.-P. and Orsted,
  B., eds.), Progress in {M}athematics, vol. 228, Birkh{\"a}user {B}oston,
  2004, pp.~1--211.

\bibitem[Kat09]{kat09}
Kato, S., \emph{An exotic {D}eligne-{L}anglands correspondence for symplectic
  groups}, Duke Math. J. \textbf{148} (2009), no.~2, 305--371.

\bibitem[Kat11]{kat11}
\bysame, \emph{Deformations of nilpotent cones and {S}pringer correspondences},
  Amer. J. Math. \textbf{133} (2011), no.~2, 519--553.

\bibitem[Kat17]{kat17}
\bysame, \emph{An algebraic study of extension algebras}, Amer. J. Math.
  \textbf{139} (2017), no.~3, 567--615.

\bibitem[Kim18a]{kim:betti}
Kim, D., \emph{On the {B}etti numbers of {S}pringer fibers for classical
  types}, Available at \url{https://arxiv.org/abs/1810.01034}. (2018), to
  appear in Transform. Groups.

\bibitem[Kim18b]{kim:total}
\bysame, \emph{On total {S}pringer representations for classical types},
  Selecta Math. (N. S.) \textbf{24} (2018), no.~5, 4141--4196.

\bibitem[Kim19]{kim:euler}
\bysame, \emph{Euler characteristic of {S}pringer fibers}, Transform. Groups
  \textbf{24} (2019), no.~2, 403--428.

\bibitem[LS85]{ls85}
Lusztig, G. and Spaltenstein, N., \emph{On the generalized {S}pringer
  correspondence for classical groups}, Algebraic {G}roups and {R}elated
  {T}opics, Adv. {S}tud. {P}ure {M}ath., vol.~6, 1985, pp.~289--316.

\bibitem[LS12]{ls12}
Liebeck, M.~W. and Seitz, G.~M., \emph{Unipotent and {N}ilpotent {C}lasses in
  {S}imple {A}lgebraic {G}roups and {L}ie {A}lgebras}, Mathematical {S}urveys
  and {M}onographs, vol. 180, American {M}athematical {S}ociety, 2012.

\bibitem[Lus81]{lus81}
Lusztig, G., \emph{Green polynomials and singularities of unipotent classes},
  Adv. Math. \textbf{42} (1981), 169--178.

\bibitem[Lus84]{lus84}
\bysame, \emph{Intersection cohomology complexes on a reductive group}, Invent.
  Math. \textbf{75} (1984), 205--272.

\bibitem[Lus86]{lus86}
\bysame, \emph{Character sheaves, {V}}, Adv. Math. \textbf{61} (1986),
  103--155.

\bibitem[Lus05]{lus05}
\bysame, \emph{Unipotent elements in small characteristic}, Transform. Groups
  \textbf{10} (2005), no.~3\&4, 449--487.

\bibitem[Mau17]{mau17}
Mautner, C., \emph{Affine pavings and the enhanced nilpotent cone}, Proc. Amer.
  Math. Soc. \textbf{145} (2017), no.~4, 1393--1398.

\bibitem[NRS18]{nrs18}
Nandakumar, V., Rosso, D., and Saunders, N., \emph{Irreducible components of
  exotic {S}pringer fibres}, J. Lond. Math. Soc. (2) \textbf{98} (2018),
  609--637.

\bibitem[Sho83]{sho83}
Shoji, T., \emph{On the {G}reen polynomials of classical groups}, Invent.
  {M}ath. \textbf{74} (1983), 239--267.

\bibitem[Sho01]{sho01}
\bysame, \emph{Green functions associated to complex reflection groups}, J.
  Algebra \textbf{245} (2001), 650--694.

\bibitem[Sho02]{sho02}
\bysame, \emph{Green functions associated to complex reflection groups,
  {I}{I}}, J. Algebra \textbf{258} (2002), 563--598.

\bibitem[Sho04]{sho04}
\bysame, \emph{Green functions attached to limit symbols}, Adv. Stud. Pure
  Math. \textbf{40} (2004), 443--467.

\bibitem[Spa82]{spa82}
Spaltenstein, N., \emph{Nilpotent classes and sheets of {L}ie algebras in bad
  characteristic}, Math. Z. \textbf{181} (1982), 31--48.

\bibitem[Spr76]{spr76}
Springer, T.~A., \emph{Trigonometric sums, {G}reen functions of finite groups
  and representations of {W}eyl groups}, Invent. Math. \textbf{36} (1976),
  173--207.

\bibitem[SS13]{ss13}
Shoji, T. and Sorlin, K., \emph{Exotic symmetric space over a finite field,
  {I}}, Transform. Groups \textbf{18} (2013), no.~3, 877--929.

\bibitem[SS14]{ss14}
\bysame, \emph{Exotic symmetric space over a finite field, {I}{I}}, Transform.
  Groups \textbf{19} (2014), no.~3, 887--926.

\bibitem[Xue12a]{xue12:comb}
Xue, T., \emph{Combinatorics of the {S}pringer correspondence for classical
  {L}ie algebras and their duals in characteristic 2}, Adv. Math. \textbf{230}
  (2012), 229--262.

\bibitem[Xue12b]{xue12}
\bysame, \emph{On unipotent and nilpotent pieces for classical groups}, J.
  Algebra \textbf{349} (2012), 165--184.

\end{thebibliography}

\end{document}